\documentclass[12pt]{article}
\usepackage{amssymb}
\usepackage{amsmath,amsthm}
\usepackage[latin1]{inputenc}
\usepackage[dvips]{graphicx}
\usepackage{hyperref}
\usepackage{color}
\usepackage{enumerate}
 \usepackage{tkz-graph}
\usepackage{tikz}

\hypersetup{colorlinks=true}

\hypersetup{colorlinks=true, linkcolor=blue, citecolor=blue,urlcolor=blue}

 \newtheorem{remark}{Remark}

 \newtheorem{lemma}[remark]{Lemma}
 \newtheorem{theorem}[remark]{Theorem}
 \newtheorem{proposition}[remark]{Proposition}
 \newtheorem{corollary}[remark]{Corollary}

\title{On the (adjacency) metric dimension of corona and strong product graphs and their local variants: combinatorial and computational results}

\author{ Juan A. Rodr\'{\i}guez-Vel\'{a}zquez \footnote{\small e-mail:\mbox{\tt juanalberto.rodriguez\@@urv.cat}} $^1$ and Henning Fernau\footnote{\small e-mail:\mbox{\tt
fernau\@@uni-trier.de}} $^2$  \\
$^1${\small Departament d'Enginyeria Inform\`{a}tica i Matem\`{a}tiques
}\\
{\small Universitat Rovira i Virgili,  Av. Pa\"{\i}sos Catalans 26, 43007
Tarragona, Spain.}
\\
$^2${\small FB 4-Abteilung Informatikwissenschaften, Universit\"{a}t Trier, 54286
Trier, Germany.}}

\newcommand{\dimension}{\operatorname{dim}}

\newcommand{\diam}{\operatorname{diam}}

\begin{document}
\maketitle

\begin{abstract}
The metric dimension is quite a well-studied graph parameter.
Recently, the adjacency metric dimension  and the local metric dimension have been introduced.
We combine these variants and introduce the local adjacency metric dimension.
We show that the (local) metric dimension of the corona product of a graph of order $n$ and some non-trivial graph $H$
equals $n$ times the  (local)  adjacency metric dimension of $H$. This strong relation also enables us to infer 
computational hardness results for computing the (local) metric dimension, based on according hardness results for 
(local) adjacency metric dimension that we also provide.
We also study combinatorial properties of the strong product of graphs and emphasize the role different types of twins 
play in determining in particular the adjacency metric dimension of a graph.
\end{abstract}

\section{Introduction}

Throughout this paper, we only consider undirected simple loop-free graphs.
We collect the standard graph-theoretic terminology at the end of this section, as well as some notions on metric spaces.

\subsection{Four notions of dimension in graphs}

Let $\mathbb{N}$ denote the set of non-negative integers. 
Given a connected graph $G=(V,E)$, we consider the function $d_G:V\times V\rightarrow \mathbb{N}$, where $d_G(x,y)$ is the length of a shortest path between $u$ and $v$. 
Clearly, $(V,d_G)$ is a metric space.
The diameter of a graph is understood in this metric. 
Alternatively, the diameter can be defined via the notion of eccentricity of a vertex, which is defined as $\varepsilon(v)=\sup\{d_G(v,u): u\in V-\{v\}\}$.
Namely, $\diam(G)=\max\{\varepsilon(v): v\in V\}$.
Similarly, the \emph{radius} of a graph is defined as $r(G)=\min \{\varepsilon(v): v\in V\}$.

A vertex set $S\subseteq V$ is said to be a \emph{metric generator} for $G$ if it is a generator of the metric space $(V,d_G)$.
A minimum metric generator is called a \emph{metric basis}, and
its cardinality the \emph{metric dimension} of $G$, denoted by $\dimension(G)$.
 Motivated by the problem of uniquely determining the location of an intruder in a network, the concept of metric
dimension of a graph was introduced by Slater in \cite{Slater1975}, where the metric generators were called \emph{locating sets}. The concept of metric dimension of a graph was also introduced by Harary and Melter in \cite{Harary1976}, where metric generators were called \emph{resolving sets}. Applications
of this invariant to the navigation of robots in networks are discussed in \cite{Khuller1996} and applications to chemistry in \cite{Johnson1993,Johnson1998}.  This graph parameter was studied further in a number
of other papers including, for instance \cite{Bailey2011,Caceres2007,Chartrand2000,Feng20121266,Guo2012raey,Haynes2006,Melter1984,Saenpholphat2004,Yero2011}.  
Several variations of metric generators including resolving dominating sets \cite{Brigham2003}, independent resolving sets \cite{Chartrand2003}, local metric sets \cite{Okamoto2010}, strong resolving sets \cite{Sebo2004}, etc. have since been introduced and studied.

 A set $S$ of vertices
in a connected graph $G$ is a \emph{local metric generator} for $G$ (also called local metric set for $G$ \cite{Okamoto2010}) if every two adjacent vertices of $G$ are distinguished by some vertex
of $S$. A minimum local metric generator is called a \emph{local metric
basis} for $G$ and its cardinality, the \emph{local metric dimension} of G, is  denoted by $\dimension_l(G)$.
 
A set $S$ of vertices
in a   graph $G$ is an \emph{adjacency generator} for $G$ (also adjacency resolving set for $G$  \cite{JanOmo2012}) if for every $x,y\in V(G)-S$ there exists $s\in S$ such that $\vert N_G(s)\cap \{x,y\} \vert =1$. A minimum adjacency generator is called an \emph{adjacency 
basis} for $G$ and its cardinality, the \emph{adjacency dimension} of $G$, is  denoted by $\dimension_A(G)$. 
These concepts were introduced in \cite{JanOmo2012} with the aim of study the metric dimension of the lexicographic product of graphs in terms of the adjacency dimension of graphs. 
Observe that an adjacency generator of a graph $G=(V,E)$ is  also a generator in a suitably chosen metric space, namely by considering $(V,d_{G,2})$,
with $d_{G,2}(x,y)=\min\{d_G(x,y),2\}$, and vice versa. 


In this paper we introduce the local adjacency dimension of a graph.  We say that a set $S$ of vertices
in a   graph $G$ is a \emph{local adjacency generator} for $G$  if for every two adjacent vertices $x,y\in V(G)-S$ there exists $s\in S$ such that $\vert N_G(s)\cap \{x,y\} \vert =1$. A minimum local adjacency generator is called a \emph{local adjacency 
basis} for $G$ and its cardinality, the \emph{local adjacency dimension} of G, is  denoted by $\dimension_{A,l}(G)$. 

\subsection{Simple facts}
\label{subsec:simpole facts}

By definition, the following inequalities hold for any graph $G$:
\begin{itemize}
 \item $\dimension(G)\leq \dimension_A(G)$;
 \item $\dimension_l(G)\leq \dimension_{A,l}(G)$;
 \item $\dimension_l(G)\leq \dimension(G)$;
 \item $\dimension_{A,l}(G)\leq \dimension_A(G)$.
\end{itemize}

Moreover, if $S$ is an adjacency generator, then at most one vertex is not dominated by $S$, so that 
$$\gamma(G)\leq \dimension_A(G)+1.$$
Namely, if $x,y$ are not dominated by $S$, then no element in $S$ distinguishes them.

We also observe that
$$\dimension_{A,l}(G)\leq\beta(G),$$
because each vertex cover is a local adjacency generator.

However, all mentioned inequalities could be either equalities or quite weak bounds.
Consider the following examples:
\begin{enumerate}
 \item $\dimension_l(P_n)=\dimension(P_n)=1  \leq \left\lfloor\frac{n}{4}\right\rfloor\le  \dimension_{A,l}(P_n)\le \left\lceil\frac{n}{4}\right\rceil\leq 
\left\lfloor \frac{2n+2}{5}\right\rfloor  =\dimension_A(P_n)$, $n\ge 7$;
 \item $\dimension_l(K_{1,n})=\dimension_{A,l}(K_{1,n})=1\leq n-1=\dimension(K_{1,n})=\dimension_A(K_{1,n})$, $n\geq 2$;
 \item $\gamma(P_n)=\left\lceil \frac{n}{3}\right\rceil\leq \left\lfloor \frac{2n+2}{5}\right\rfloor =\dimension_A(P_n)$, $n\ge 7$;
 \item $\left\lfloor\frac{n}{4}\right\rfloor\le  \dimension_{A,l}(P_n)\le \left\lceil\frac{n}{4}\right\rceil\leq \left\lfloor \frac{n}{2}\right\rfloor=\beta(P_n)$,   $n\geq 2$.
\end{enumerate}

\subsection{Our main results}

In this paper we study the (local) metric dimension of corona product graphs via the (local) adjacency dimension of a graph.
We show that the (local) metric dimension of the corona product of a graph of order $n$ and some non-trivial graph $H$
equals $n$ times the  (local)  adjacency metric dimension of $H$. This  relation is much stronger and under weaker conditions compared to the results of
 Jannesari and Omoomi \cite{JanOmo2012}
concerning  the lexicographic product of graphs. This 
also enables us to infer 
NP-hardness results for computing the (local) metric dimension, based on according NP-hardness results for 
(local) adjacency metric dimension that we also provide.
The relatively simple reductions also allow us to conclude hardness results based on the Exponential Time Hypothesis.  
We also study combinatorial properties of the strong product of graphs and emphasize the role different types of twins 
play in determining in particular the adjacency metric dimension of a graph.

\subsection{Some notions from graph theory and topology}

In this paragraph, we collect some standard graph-theoretic terminology that we employ.
As usual, graphs are specified like $G=(V,E)$, where $V$ is the set of vertices and $E$ is the set of edges of the graph $G$.
$|V|$ is also known as the \emph{order} of $G$. 
Two vertices $u,v\in V$ with an edge between them, i.e., $uv\in E$, are also called \emph{adjacent} or \emph{neighbors},
and this is also written as 
$u\sim v$. For a  vertex
$v$ of $G$, $N_G(v)$ denotes the set of neighbors that $v$ has in $G$, i.e.,  $N_G(v)=\{u\in V:\; u\sim v\}$. 
The set $N_G(v)$ is called the \emph{open neighborhood of} $v$ in $G$ and $N_G[v]=N_G(v)\cup \{v\}$ is called the \emph{closed neighborhood of} $v$ in $G$.  
A vertex set $D\subseteq V$ is called a \emph{dominating set} if $\bigcup_{v\in D} N_G[v]=V$. The \emph{domination number} of $G$,
denoted by  $\gamma (G)$, is the minimum cardinality among all dominating sets in $G.$ 
A vertex set $I\subseteq V$ is called an \emph{independent set} if for all $u,v\in I$, $uv\notin E$. 
The \emph{independent set number} of $G$,
denoted by  $\alpha (G)$, is the maximum cardinality among all independent sets in $G.$ 
The difference between the order and the independent set number of a graph $G$  is also known as the \emph{vertex cover number}
of $G$, written $\beta(G)$, as the complement of an independent set is called a \emph{vertex cover}.

Given a set $S\subseteq V$, we denote by $\langle S\rangle_G$ the subgraph of $G$ induced by $S$, omitting the subscript $G$ if clear from the context. 
In particular, if $S=\{x\}$ we will use the notation $\langle x\rangle$ instead of $\langle \{x\}\rangle$.
A graph is \emph{empty} if it contains no edges.
A graph $G=(V,E)$ is \emph{bipartite} if $V$ can be partitioned into two sets $V_1$ and $V_2$ such that
both $\langle V_1\rangle_G$ and $\langle V_2\rangle_G$ are empty graphs.
Two vertices $u,v$ are \emph{connected} if there is a sequence of vertices $$u=v_1,v_2,v_3,\ldots, v_r=v$$
such that $v_i\in N_G[v_{i+1}]$ for all $i=1,\ldots,r-1$. Connectedness defines an equivalence relation on $V$,
and the equivalence classes are known as the \emph{connected component}s of $G$. Mostly, they are identified with the graphs they induce.
A graph is \emph{connected} if it has only one connected component.

For building examples, we also make use of well-known abbreviations for typical graphs, like
\begin{itemize}
 \item $P_n$: the \emph{path} on $n$ vertices;
\item $C_n$: the \emph{cycle} on $n$ vertices (with $n\geq 3$);
\item $K_n$: the \emph{complete graph} on $n$ vertices.
\item $K_{r,s}$ is the \emph{complete bipartite graph} with $r$ vertices on one side and $s$ vertices on the other.
\item $W_n$ is the \emph{wheel graph} that can be described as $K_1+C_{n-1}$ (with $n\geq 4$).
\item $F_n$ is the \emph{fan graph} that can be described as $K_1+P_{n-1}$ (with $n\geq 3$).
\item $N_n$ is the \emph{null graph} (or  empty graph) that can be described as the complement of  $K_n$, \textit{i.e}., $N_n$ consists of $n$ isolated nodes with no edges.
\item 
$K_1=P_1=N_1$ is also known as the \emph{trivial graph}.
\end{itemize}

Let $\mathbb{R}_{\geq 0}$ denote  the set of non-negative real numbers. 
A \emph{metric space} is a pair $(X,d)$, where $X$ is a set of points and $d:X\times X\to \mathbb{R}_{\geq 0}$ satisfies
$d(x,y)=0$ if and only if $x=y$, $d(x,y)=d(y,x)$  for all $x,y \in X$ and $d(x,y)\le d(x,z)+d(z,y)$  for all $x,y,z\in X$.
The \emph{diameter} of a point set $S\subseteq X$ is defined as $\diam(S)=\sup\{d(x,y):x,y\in S\}$.
A \emph{generator} of a metric space $(X,d)$ is a set $S$ of points in the space  with the property that every point of the space is uniquely determined by the distances from the elements of $S$. 
A point $v\in X$ is said to \emph{distinguish} two points $x$ and $y$ of $X$ if $d(v,x)\ne d(v,y)$.
Hence, $S$ is a generator if and only if any pair of points of $X$ is distinguished by some element of $S$.

We conclude this section by giving the definitions of the graph operations that we examine, starting with the better known operations and moving on to the less known ones that are yet more
important in this paper.  
Let $G$ and $H$ be two graphs of order $n$ and $n'$, respectively. 
\begin{itemize}
 \item The \emph{complement (graph)} $\overline{G}$ of $G$ has the same vertex set as $G$, but an edge between two distinct vertices $x,y$ if and only if $x\notin N_G(y)$.

\item The \emph{graph union} $G\cup H$ is defined if the vertex sets $V(G)$ and $V(H)$ are disjoint and then refers to
the graph $(V(G)\cup V(H),E(G)\cup E(H))$. 
\item The \emph{join (graph)} $G+H$ is defined as the graph obtained from vertex-disjoint graphs $G$ and $H$ by taking one copy of $G$ and one copy of $H$ and 
joining by an edge each vertex of $G$ with each vertex of $H$. 
Also, the complete graph $K_n$ can be recursively described as $K_1+K_{n-1}$.
\item The \emph{corona product (graph)} $G\odot H$ is defined as 
the graph obtained from $G$ and $H$ by taking one copy of $G$ and $n$ copies of $H$ and joining by an edge each vertex from the $i^{th}$ copy of $H$ with the $i^{th}$ vertex of $G$ \cite{Frucht1970}. 
We will denote by $V=\{v_1,v_2,\ldots,v_n\}$ the set of vertices of $G$ and by $H_i=(V_i,E_i)$ the $i^{th}$ copy of $H$ such that $v_i\sim x$ for every $x\in V_i$. 
Notice that the corona graph $K_1\odot H$ is isomorphic to the join graph $K_1+H$. 
\item 
The {\em strong product (graph)} $G\boxtimes H$ of two graphs $G=(V_{1},E_{1})$ and $H=(V_{2},E_{2})$ is the graph with vertex set $V\left(G\boxtimes H\right)=V_{1}\times V_{2}$, 
where two distinct vertices $(x_{1},x_{2}),(y_{1},y_{2})\in V_{1}\times V_{2}$ are adjacent in $G\boxtimes H$ if and only if one of the following holds.
\begin{itemize}
\item $x_{1}=y_{1}$ and $x_{2}\sim y_{2}$, or
\item $x_{1}\sim y_{1}$ and $x_{2}=y_{2}$, or
\item  $x_{1}\sim y_{1}$ and $x_{2}\sim y_{2}$.
\end{itemize}
Alternatively,  two distinct vertices $(x_{1},x_{2}),(y_{1},y_{2})$ of  $G\boxtimes H$  are adjacent if and only if $x_1\in N_G[y_1]$ and $x_2\in N_H[y_2]$. 
\end{itemize}
For our computational complexity results, it is important but easy to observe that all these graph operations can be performed in polynomial time, given one or 
two input graphs.

\section{The metric dimension of corona product graphs versus  the adjacency dimension of a graph}

\subsection{Computing the metric dimension of corona graphs with the adjacency dimension of the second operand}

The following is the first main combinatorial result of this paper and provides a strong link between the metric dimension of the corona product of two graphs and 
the adjacency dimension of the second graph involved in the product operation.

\begin{theorem} \label{mainTheoremDim}
For any connected graph $G$ of order $n\ge 2$ and any non-trivial graph $H$, 
$$\dimension(G\odot H)=n \cdot \dimension_A(H).$$
\end{theorem}

\begin{proof}
We first need to prove that $\dimension(G\odot H)\le n\cdot \dimension_A(H)$.
For any $i\in \{1,\ldots,n\}$, let $S_i$ be an adjacency  basis of $H_i$, the $i^{th}$-copy of $H$.  
In order to show that $X:=\bigcup_{i=1}^n S_i$ is a  metric generator for $G\odot H$,
 we differentiate the following four cases for two vertices $x,y\in V(G\odot H)-X$.
\begin{enumerate}
 \item 
$x,y\in V_i$. Since $S_i$ is an adjacency basis of $H_i$, there there exists a vertex $u\in S_i$ such that $\vert N_{H_i}(u)\cap \{x,y\} \vert =1$.
Hence,
$$d_{G\odot H}(x,u)=d_{\langle v_i \rangle +H_i}(x,u)\ne d_{\langle v_i \rangle +H_i}(y,u)= d_{G\odot H}(y,u).$$

\item 
$x\in V_i$ and $y\in V$. If $y=v_i$, then for $u\in S_j$, $j\ne i$, we have 
$$d_{G\odot H}(x,u)=d_{G\odot H}(x,y)+d_{G\odot H}(y,u)> d_{G\odot H}(y,u).$$ Now, if $y=v_j$, $j\ne i$, then we also take $u\in S_j$ and we proceed as above. 

\item 
$x=v_i$ and $y=v_j$. For $u\in S_j$,  we find that  $$d_{G\odot H}(x,u)=d_{G\odot H}(x,y)+d_{G\odot H}(y,u)> d_{G\odot H}(y,u).$$

\item 
$x\in V_i$ and $y\in V_j$, $j\ne i$. In this case, for  $u\in S_i$ we have 
$$d_{G\odot H}(x,u)\le 2<3\le  d_{G\odot H}(u,y).$$

\end{enumerate}

Hence, $X$ is a metric generator for $G\odot H$ and, as a consequence, $$\dimension(G\odot H)\le \sum_{i=1}^n \vert S_i\vert= n \cdot \dimension_A(H).$$

It remains to prove that $\dimension(G\odot H)\ge n\cdot \dimension_A(H)$. To do this, let $W$ be  a  metric basis for $G\odot H$ and, for any $i\in \{1,\ldots,n\}$, let $W_i:=V_i\cap W$.
Let us show that $W_i$ is an adjacency metric generator for $H_i$. To do this, consider two different vertices $x,y\in V_i-W_i$. 
Since no  vertex $a\in V(G\odot H)-V_i$ distinguishes the pair $x,y$, there exists some $u\in W_i$ such that 
$d_{G\odot H}(x,u)\ne d_{G\odot H}(y,u)$. Now, since   $d_{G\odot H}(x,u)\in \{1,2\}$ and $d_{G\odot H}(y,u)\in \{1,2\}$, we conclude that $\vert  N_{H_i}(u)\cap \{x,y\}\vert=1$ 
and consequently, $W_i$ must be an adjacency  generator for 
$H_i$. 
 Hence, for any $i\in \{1,\ldots,n\}$, $|W_i|\ge \dimension_A(H_i)$.
 Therefore,   $$\dimension (G\odot H)= |W|\ge \sum_{i=1}^n |W_i|\ge \sum_{i=1}^n \dimension_A( H_i)=n\cdot \dimension_A(H).$$ This completes the proof.
\end{proof}

\subsection{Consequences of 
Theorem \ref{mainTheoremDim}}

Theorem \ref{mainTheoremDim} allows us to investigate   $\dimension(G\odot H)$ through the study of $\dimension_A(H)$, and vice versa.

\begin{theorem}{\rm \cite{Yero2011}} \label{KnownDimCorona}
Let $G$ be a connected graph of order $n$ and let $H$ be some graph. 
\begin{enumerate}[{\rm (i)}]
\item If  $\diam(H) \le 2$, then $\dimension(G\odot H)=n \cdot \dimension(H).$

\item If  $\diam(H) \ge 6$ or $H$ is a cycle graph of order at least $7$, then $$\dimension(G\odot H)=n \cdot \dimension(K_1+H).$$
\end{enumerate}
\end{theorem}

As a direct consequence of Theorems \ref{mainTheoremDim} and 
\ref{KnownDimCorona} we obtain the following result. 

\begin{proposition}\label{PropDiam2andCycle}
Let $H$ be a graph. 
\begin{enumerate}[{\rm (i)}]
\item If  $\diam(H) \le 2$, then $\dimension(H)= \dimension_A(H).$

\item If  $\diam(H) \ge 6$ or $H$ is a cycle graph of order at least $7$, then $$\dimension(K_1+H)= \dimension_A(H).$$
\end{enumerate}
\end{proposition}

In particular, it was shown in \cite{Buczkowski2003} that for any wheel graph $W_{r+1}$ and any fan graph $F_{r+1}$,
 $r\ge 7$, it holds  that $\dimension(W_{r+1})=\dimension(F_{r+1})=\left \lfloor  \frac{2r+2}{5}\right\rfloor$. As $W_{r+1}=K_1+C_r$ and $F_{r+1}=K_1+P_r$,  it holds that 
$\dimension_A(C_r)=\dimension_A(P_r)=\left \lfloor  \frac{2r+2}{5}\right\rfloor$ for any $r\ge 7$.

\begin{theorem}{\rm \cite{JanOmo2012}}\label{ThAdjacencyDim}
For any graph $H$ of order $n'\ge 2$,
\begin{enumerate}[{\rm (i)}]
\item $\dimension_A(H)=\dimension_A(\overline{H})$. 
\item $\dimension_A(H)=1$ if and only if $H\in \{P_2,P_3,\overline{P_2},\overline{P_3}\}$.
\item $\dimension_A(H)=n'-1$ if and only if $H\cong K_{n'}$ or $H\cong \overline{K}_{n'}$.
\end{enumerate}
\end{theorem}

The following result is a direct consequence of Theorems \ref{mainTheoremDim} and \ref{ThAdjacencyDim}.

\begin{proposition}
For any connected graph $G$ of order $n\ge 2$ and any graph $H$ of order $n'\ge 2$,
\begin{enumerate}[{\rm (i)}]
\item $\dimension (G\odot H)=\dimension(G\odot \overline{H})$. 
\item $\dimension(G\odot H)=n$ if and only if $H\in \{P_2,P_3,\overline{P_2},\overline{P_3}\}$.
\item $\dimension(G \odot H)=n(n'-1)$ if and only if $H\cong K_{n'}$ or $H\cong \overline{K}_{n'}$.
\end{enumerate}
\end{proposition}

\subsection{A detailed analysis of the adjacency dimension of the corona product via 
the  adjacency dimension of the second operand}

We now analyze the adjacency dimension of the corona product $G\odot H$ in terms of the adjacency dimension of $H$.
In particular, we show that for any connected  graph $G$ of order $n\ge 2$ and any non-trivial graph $H$,
$$n-1\geq \dimension_A(G\odot H)-n\cdot \dimension_A(H)\geq 0.$$
The bounds in the inequalities are attained in very specific situations which we are going to characterize.

\begin{theorem}\label{TheOnlyPosibilitiesDimAdjCorona(a)}
Let $G$ be a  connected graph of order $n\ge 2$ and let $H$ be a non-trivial graph. 
If there exists an adjacency basis $S$ for $H$ which is also a dominating set, and if for every $v\in V(H)-S$, it is satisfied that $S\not\subseteq N_H(v)$, then  
$$\dimension_A(G\odot H)=n\cdot \dimension_A(H).$$
\end{theorem}

\begin{proof}
Suppose that $S$ is an adjacency basis for $H$ which is also a dominating set.  
Let $S_i$ be the copy of $S$ in the $i^{th}$ copy of $H$ in $G\odot H$. First of all, note that by Theorem~\ref{mainTheoremDim} we have
$$\dimension_A(G\odot H)\ge \dimension(G\odot H)=n\cdot \dimension_A(H).$$

Suppose that for every $v\in V(H)-S$ it is satisfied that $S\not\subseteq N_H(v)$.
We claim that $\dimension_A(G\odot H)\le n\vert S\vert$. To see this, let $S'=\bigcup_{i=1}^nS_i$ and let us prove that $S'$ is an adjacency generator for $G\odot H$. 
So we differentiate the following cases for any pair $x,y$ of vertices of $G\odot H$  not belonging to $S'$. 
\begin{enumerate}
 \item $x,y\in V_i$. Since $S_i$ is an adjacency basis of $H_i$,  there exists $u_i\in S_i$ such that either $u_i\sim x$ and $u_i \not\sim y$ or $u_i\not\sim x$ and $u_i \sim y$. 

\item $x\in V_i$, $y\in V_j$, $j\ne i.$ As $S_i$ is a dominating set of $H_i$, there exists $u\in S_i$ such that $u\sim x$ and, obviously, $u\not \sim y$.
\item $x\in V_i$, $y=v_i\in V$. By assumption, we have that $S_i\not\subseteq N_{H_i}(x)$, so for every $u\in S_i-N_{H_i}(x)$,  we find that $u\sim y$.

\item $x\in V_i$, $y=v_l\in V$, $i\ne l$. In this case for every $u\in S_l$, we have $u\sim y $ and $u\not \sim x$.
\item  $x=v_i,y=v_j\in V$, $i\ne j$. Taking $u\in S_i$, we have $u\sim x $ and $u\not \sim y$.

\end{enumerate}

From the  cases above, we conclude that $S'$ is an adjacency generator for $G\odot H$ and, as a consequence,
  $\dimension_A(G\odot H)\le \vert S' \vert = n\cdot \vert S\vert=n\cdot \dimension_A(H)$. 
\end{proof}

\begin{corollary}
\label{Cor-TheOnlyPosibilitiesDimAdjCorona(a)}
Let $r\ge 7$ be an integer such that $r\not \equiv 1 \bmod 5$ and  $r\not\equiv 3 \bmod 5$. For any connected graph $G$   of order $n\ge 2$,  
 $$\dimension_A(G\odot C_r)=\dimension_A(G\odot P_r)=n \cdot \left\lfloor \frac{2r+2}{5}\right\rfloor.$$
\end{corollary}

\begin{proof}We shall construct an adjacency basis of $C_r$ (and also of $P_r$), say $S_r$,  which must satisfy the premisses of Theorem \ref{TheOnlyPosibilitiesDimAdjCorona(a)}.
Notice that as a consequence of Proposition \ref{PropDiam2andCycle} we previously  showed that $\dimension_A(C_r)=\dimension_A(P_r)=\left \lfloor  \frac{2r+2}{5}\right\rfloor$. So, the cardinality of $S_r$ must be $\left \lfloor  \frac{2r+2}{5}\right\rfloor$.   
 Let $V_r=\{0,\dots,r-1\}$ be the set of vertices of the cycle $C_r$ (or of the path $P_r$, respectively). Define
$$S_r=\left\{j\in V_r\mid 1\equiv j\bmod 5\lor  3\equiv j\bmod 5\right\}.$$
It is easy to verify that for $r\not \equiv 1 \bmod 5$ and  $r\not\equiv 3 \bmod 5$ the set  $S_r$ is an adjacency generator for $C_r$ (and of $P_r$)   that is also a dominating set. 
Finally, it is clear that since $r\ge 7$, for every vertex of $H\in \{C_r,P_r\}$ we have $S_r\not\subseteq N_H(v)$, as $\vert N_H(v) \vert \le 2$ and $\vert S_r\vert \ge 3$.
\end{proof}

It is instructive to notice that $S_r$ (as defined in the previous proof) is also an adjacency generator  of the cycle $C_r$ and also  of the path $P_r$ if  $r\equiv 1 \bmod 5$, but in that case, it fails to be a dominating set,
while $S_r$ is a dominating set of $C_r$ and of $P_r$ that fails to be an adjacency generator if  $r\equiv 3 \bmod 5$.

\begin{theorem}\label{TheOnlyPosibilitiesDimAdjCorona(b)}
Let $G$ be a  connected graph of order $n\ge 2$ and let $H$ be a non-trivial graph.  
If there exists an adjacency basis for $H$ which is also a dominating set and if, for any adjacency basis $S$ for $H$,  there exists $v\in V(H)-S$ such that $S\subseteq N_H(v)$, 
then  $$\dimension_A(G\odot H)=n\cdot \dimension_A(H)+\gamma (G).$$
\end{theorem}

\begin{proof} 
Let $W$ be an adjacency basis for $G\odot H$ and let $W_i=W\cap V_i$ and $U=W\cap V$. Since two vertices belonging to $V_i$ are not distinguished by any $u\in W-V_i$, the set $W_i$ must be  an adjacency generator for $H_i$. Now consider the partition $\{V',V''\}$ of $ V$   defined as follows:
$$V'=\{v_i\in V: \; \; \; \vert W_i\vert = \dimension_A(H)\}\; \; {\rm  and }\;\;
V''=\{v_j\in V: \; \; \;\vert W_j\vert \ge  \dimension_A(H)+1\}.$$

 Note that, if $v_i\in V'$, then  $W_i$ is an adjacency basis for $H_i$, thus in this case  there exists $u_i\in V_i$ such that $W_i\subseteq N_{H_i}(u_i)$. Then  the pair $u_i,v_i$ is not distinguished by the elements of $W_i$ and, as a consequence, either $v_i\in U$ or there exists $v_j\in U$ such that $v_j\sim v_i.$  Hence, 
 $U\cup V''$ must be a dominating set and, as a result, $$\vert U\cup V''\vert \ge \gamma(G).$$ 
 So we obtain the following:
\begin{eqnarray*}
\dimension_A(G\odot H)&=&\vert W\vert \\
             &= & \bigcup_{v_i\in V'} \vert W_i \vert + \bigcup_{v_j\in V''} \vert W_j\vert +\left\vert  U  \right\vert \\
             &\ge  & \sum_{v_i\in V'} \dimension_A(H) + \sum_{v_j\in V''} (\dimension_A(H)+1) +\left\vert  U  \right\vert \\
             & =  & n\cdot  \dimension_A(H) + \vert V''\vert +\left\vert  U  \right\vert \\
             &\ge & n\cdot  \dimension_A(H) +\vert  V'' \cup  U  \vert \\   
             &  \ge & n\cdot \dimension_A(H)+\gamma(G).
\end{eqnarray*}

To conclude the proof,  we consider    an adjacency basis $S$ for $H$ which is also  a dominating set, and we denote by  $S_i$   
the copy of $S$ corresponding to $H_i$.  We  claim that for any dominating set $D$ of $G$ of minimum cardinality $\vert D\vert=\gamma(G)$,  
the set $D\cup (\bigcup_{i=1}^n S_i)$ is an adjacency generator for $G\odot H$ and, as a result, 
$$\dimension_A(G\odot H)\le \left\vert D\cup \left(\bigcup_{i=1}^n S_i\right)\right \vert =n\cdot \dimension_A(H)+\gamma(G).$$
To see this, we differentiate the same cases as in the proof of Theorem \ref{TheOnlyPosibilitiesDimAdjCorona(a)} 
with the only difference that now in Case 3 either $y=v_i\in D$ or there exists some $v_j\in D$ such that $v_j\sim y$. Of course, if $y=v_i\not \in D$, then    $v_j$ distinguishes the pair $x,y$. Therefore, the result follows. 
\end{proof}

\begin{corollary}Let $r\geq 2$. 
For any connected graph  $G$  of order $n\ge 2$, 
 $$\dimension_A(G\odot K_{r})=n(r-1)+\gamma(G).$$
\end{corollary}

\begin{theorem}\label{TheOnlyPosibilitiesDimAdjCorona(c)}
Let $G$ be a  connected graph of order $n\ge 2$ and let $H$ be a non-trivial graph. 
If no adjacency basis  for $H$ is a dominating set, then  $$\dimension_A(G\odot H)=n\cdot \dimension_A(H)+n-1.$$
\end{theorem}

\begin{proof}
We assume that no adjacency basis  for $H$ is a dominating set.
As explained in Subsection~\ref{subsec:simpole facts},
%
 if $B$ is an  adjacency basis  for $H$ which is not a  dominating set, then there exists exactly one vertex of $H$ which is not dominated by $B$.

As in the proof of the previous theorem, we take  $W$ as an adjacency basis for $G\odot H$ and we deduce that every  $W_i=W\cap V_i$ must be an adjacency generator for $H_i$. 
So, for any $W_i$ which is not an adjacency basis for $H_i$ we have $\vert W_i\vert \ge \dimension_A(H)+1$. Also, 
for any pair $W_i$, $W_j$ which are adjacency bases for $H_i$ and $H_j$, there exist  two vertices $w_i\in V_i-W_i$ and $w_j\in V_j-W_j$ which are not   
dominated by the elements of $W_i$ and $W_j$, respectively. Then, $v_i$ or $v_j$ must belong to $W$. Hence, if $W_{l_1},W_{l_2},\dots,W_{l_k}$ are adjacency bases for 
$H_{l_1},H_{l_2},\dots,H_{l_k}$, respectively, then 
$\vert \{v_{l_1},v_{l_2},\dots,v_{l_k}\}\cap W\vert \ge k-1,$ and, as a consequence,
\begin{eqnarray*}
\dimension_A(G\odot H)&=&\vert W\vert\\
&=& \vert V \cap W\vert+\sum_{i=1}^k\vert W_{l_i} \vert +\sum_{j\not\in \{{l_1},\dots,{l_k}\}} \vert W_{l_j} \vert \\
&\ge& (k-1)+k\cdot \dimension_A(H)+(n-k)(\dimension_A(H)+1)\\
&=& n\cdot \dimension_A(H)+n-1.
\end{eqnarray*}

Now we claim that for any adjacency basis $B$ of $H$ the set $B'=(V-\{v_n\})\cup (\bigcup_{i=1}^{n}B_i)$ is an adjacency generator for $G\odot H$, where $B_i$ is the copy of $B$ corresponding to the graph $H_i$. 
To see this we differentiate some cases for $x,y\not \in B'$. 
If $x,y\in V_i$, then there exists $b_i\in B_i$ which distinguishes them. If  $x \in V_i$ and $ y\in V_j$, for $i<j$, then $v_i\in B'$ satisfies $v_i\sim x$ and $v_i\not \sim y$. Finally, if $x=v_n$, then the pair $x,y$ is distinguished by $v\in N_G(v_n)\subset B'$, when $y\in V_n$, and by $b_n\in B_n\subset S$, when $y\not \in V_n$.
Hence, $B'$ is an adjacency generator for $G\odot H$ and, as a result,
$$\dimension_A(G\odot H)\le \vert B'\vert= n\cdot \dimension_A(H)+n-1.$$
Therefore, the proof is complete.  
\end{proof}

It is easy to check that any adjacency basis of a star graph $K_{1,r}$ is composed of $r-1$ leaves. 
This will leave the last leaf non-dominated. 
Thus, Theorem \ref{TheOnlyPosibilitiesDimAdjCorona(c)} leads to the following result. 

\begin{corollary}\label{Cor-TheOnlyPosibilitiesDimAdjCorona(c)} 
For any connected graph  $G$  of order $n\ge 2$, 
 $$\dimension_A(G\odot K_{1,r})=n\cdot r-1.$$
\end{corollary}

Given a vertex $v\in V$ we denote by $G-v$ the subgraph obtained from  $G$ by removing $v$ and the edges incident with it. 
We define the following auxiliary domination parameter:
$$\gamma'(G):=\min_{v\in V(G)} \{\gamma(G-v)\}.$$

\begin{theorem}\label{TheOnlyPosibilitiesDimAdjCorona(d)}
Let $H$ be a non-trivial graph such that some of its adjacency bases  are also  dominating
sets, and some are not. If there exists an adjacency basis $S'$ for $H$ such that for every $v\in V(H)-S'$ it is satisfied that $S'\not \subseteq N_H(v) $, 
and for any adjacency basis $S$ for $H$ which is also a dominating set,   there exists some $v\in V(H)-S$ such that $S\subseteq N_H(v)$, then  
for any   connected graph $G$ of order $n\ge 2$,   $$\dimension_A(G\odot H)=n\cdot \dimension_A(H)+\gamma' (G).$$
\end{theorem}

\begin{proof}Assume that 
for any adjacency basis $S$ for $H$ which is also a dominating set,   there exists $v\in V(H)-S$ such that $S\subseteq N_H(v)$. Also, assume that   there exists an adjacency basis $S'$ for $H$ such that for every $v\in V(H)-S'$ it is satisfied that $S'\not \subseteq N_H(v) $. 
Let $S_i$  be the copy of $S$  corresponding to  $H_i$ and, analogously,  let $S'_j$  be the copy of $S'$  corresponding to  $H_j$.

Let $V(G)=\{v_1,\dots,v_n\}$. 
We suppose, without loss of generality, that 
$\gamma'(G)=\gamma(G-v_n)=\vert D\vert,$ where $D$ is a dominating set of $G-v_n$. 
We claim that  $X=D\cup S'_n\cup \left(\bigcup_{i=1}^{n-1}S_i\right)$ is an adjacency generator for $G\odot H$. To show it, we differentiate the following cases for any pair $x,y$ of vertices of $G\odot H$  not belonging to $X$. 
\begin{enumerate}
 \item $x,y\in V_i$. Suppose $i\ne n$. Since $S_i$ is an adjacency basis of $H_i$,  there exists $u_i\in S_i$ such that either $u_i\sim x$ and $u_i \not\sim y$ or $u_i\not\sim x$ and $u_i \sim y$. Analogously, for $i=n$  there exists $u_n\in S'_n$ which differentiates the pair $x,y$.    

\item  $x\in V_i$, $y\in V_j$, $j> i.$ As $S_i$ is a dominating set of $H_i$, there exists $u\in S_i$ such that $u\sim x$ and, obviously, $u\not \sim y$.

\item $x\in V_i$, $y=v_i\in V$. Let $i=n$. If $x$ is dominated by $S'_n$, then by assumption we have that $S'_n\not\subseteq N_{H_n}(x)$, so  every $u\in S'_n-N_{H_n}(x)$  distinguishes  $x$ and $y$. Also, if $x$ is not dominated by $S'_n$, then 
for every $u\in S'_n$ we have $u\sim y=v_n$ and  $u\not \sim x$. For $i\ne n$ we have that either $v_i\in D$ or $v_i\sim v_j$, for some $v_j\in D$. Obviously, if $v_i\not\in D$, then $v_j$ distinguishes the pair $x,y.$

\item  $x\in V_i$, $y=v_l\in V $, $i\ne l$. If $l\ne n$, then for every $u\in S_l$ we have $u\sim y $ and $u\not \sim x$. Analogously, if $l=n$, then for every $u\in S'_n$ we have $u\sim y $ and $u\not \sim x$.
 
 \item $x=v_i,y=v_j\in V $, $i< j$. Taking $u\in S_i$ we have $u\sim x $ and $u\not \sim y$.

\end{enumerate}



From the  cases above we conclude that $X$ is an adjacency generator for $G\odot H$ and, as a consequence,  
$$\dimension_A(G\odot H)\le \vert X \vert= n\cdot \dimension_A(H)+\gamma' (G).$$ 

To conclude the proof we need to prove that $\dimension_A(G\odot H)\ge \dimension_A(H)+\gamma' (G).$
Let $W$ be an adjacency basis for $G\odot H$ and let $W_i=W\cap V_i$ and $U=W\cap V$. We know that since two vertices belonging to $V_i$ are not distinguished by any $u\in W-V_i$, the set $W_i$ must be  an adjacency generator for $H_i$. Now consider the partition $\{V',V'', V'''\}$  of $V$   defined as follows:
$V'$ is composed of the vertices $v_i$ of $G$ such that $W_i$ is an adjacency basis but it is not a dominating set of $H_i$, $V''$ is composed of the vertices $v_i$ of $G$ such that $W_i$ is an adjacency basis and also it is dominating set of $H_i$  and finally  $V'''$ is composed of the vertices $v_i$ of $G$ such that $W_i$ is not an adjacency basis for $H_i$.

 Note that, if $v_i,v_j\in V'$, then there exist  two vertices $w_i\in V_i-W_i$ and $w_j\in V_j-W_j$ which are not   dominated by the elements of $W_i$ and $W_j$, respectively. Then, $v_i$ or $v_j$ must belong to $U$ and, as a consequence, $\vert U\cap V'\vert \ge \vert V'\vert -1.$ Now, 
 if $v_i\in V''$, then   there exists $u_i\in V_i$ such that $W_i\subset N_{H_i}(u_i)$. Then  the pair $u_i,v_i$ is not distinguished by the elements of $W_i$ and, as a consequence, either $v_i\in U$ or there exists $v_j\in U$ such that $v_j\sim v_i.$  Hence, at most one vertex of $G$ is not dominated by $U\cup V'''$ and, as a result, $$\vert U\cup V'''\vert \ge \gamma'(G).$$
 
 So we have the following:
\begin{eqnarray*}
\dimension_A(G\odot H)&=&\vert W\vert \\
             &= & \bigcup_{v_i\in V'\cup V''} \vert W_i \vert + \bigcup_{v_j\in V'''} \vert W_j\vert +\left\vert  U  \right\vert \\
             &\ge & \sum_{v_i\in V'\cup V''} \dimension_A(H) + \sum_{v_j\in V'''} (\dimension_A(H)+1) +\vert  U  \vert \\
             &= & n\cdot \dimension_A(H)+\vert V'''\vert +\vert  U  \vert \\
              &\ge & n\cdot \dimension_A(H) +\vert  V''' \cup  U  \vert \\   
              & \ge & n\cdot \dimension_A(H)+\gamma'(G).
\end{eqnarray*}
Therefore, the result follows.  
\end{proof}

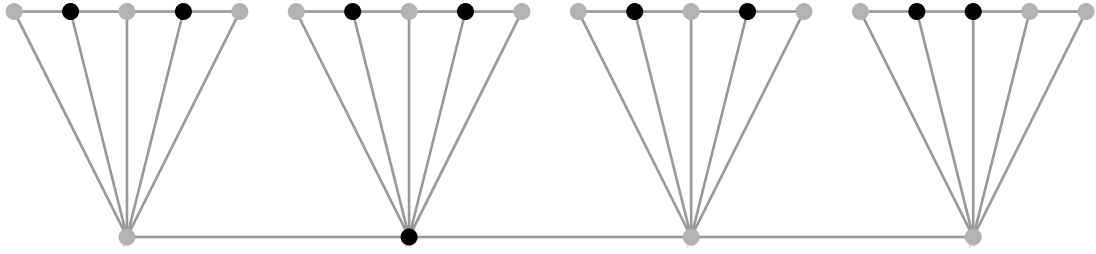
\begin{figure}[h]
\begin{center}
\begin{tikzpicture}
[inner sep=0.7mm, place/.style={circle,draw=black!30,fill=black!30,thick},xx/.style={circle,draw=black!99,fill=black!99,thick},
transition/.style={rectangle,draw=black!50,fill=black!20,thick},line width=1pt,scale=0.5]
\coordinate (AA) at (-17,6);
\coordinate (A) at (-15.5,6);
\coordinate (B) at (-14,6);
\coordinate (C) at (-12.5,6);
\coordinate (D) at (-11,6);
\coordinate (EE) at (-9.5,6);
\coordinate (E) at (-8,6);
\coordinate (F) at (-6.5,6);
\coordinate (G) at (-5,6);
\coordinate (H) at (-3.5,6);
\coordinate (II) at (-2,6);
\coordinate (I) at (-0.5,6);
\coordinate (J) at (1,6);
\coordinate (K) at (2.5,6);
\coordinate (L) at (4,6);
\coordinate (MM) at (5.5,6);
\coordinate (M) at (7,6);
\coordinate (N) at (8.5,6);
\coordinate (O) at (10,6);
\coordinate (P) at (11.5,6);
\coordinate (Q) at (-14,0);
\coordinate (R) at (-6.5,0);
\coordinate (S) at (1,0);
\coordinate (T) at (8.5,0);
\draw[black!40] (Q)--(AA)--(A)--(B) -- (C) -- (D) -- (Q)--(C);
\draw[black!40] (A)--(Q)--(B);

\draw[black!40] (R)--(EE)--(E)--(F) -- (G) -- (H) -- (R)--(G);
\draw[black!40] (E)--(R)--(F);

\draw[black!40] (S)--(II)--(I)--(J) -- (K) -- (L) -- (S)--(K);
\draw[black!40] (I)--(S)--(J);

\draw[black!40] (T)--(MM)--(M)--(N) -- (O) -- (P) -- (T)--(O);
\draw[black!40] (M)--(T)--(N);

\draw[black!40] (Q)--(R)--(S)--(T);

\node at (AA) [place]  {};
\node at (A) [xx]  {};
\node at (B) [place]  {};
\node at (C) [xx]  {};
\node at (D) [place]  {};
\node at (EE) [place]  {};
\node at (E) [xx]  {};
\node at (F) [place]  {};
\node at (G) [xx]  {};
\node at (H) [place]  {};
\node at (II) [place]  {};
\node at (I) [xx]  {};
\node at (J) [place]  {};
\node at (K) [xx]  {};
\node at (L) [place]  {};
\node at (MM) [place]  {};
\node at (M) [xx]  {};
\node at (N) [xx]  {};
\node at (O) [place]  {};
\node at (P) [place]  {};
\node at (Q) [place]  {};
\node at (R) [xx]  {};
\node at (S) [place]  {};
\node at (T) [place]  {};
\end{tikzpicture}
\end{center}
\caption{\label{fig-P5}The bold type indicates an adjacency  basis for $P_4\odot  P_5$ which is not a dominating set. 
 Since $H=P_5$ satisfies the premises of Theorem \ref{TheOnlyPosibilitiesDimAdjCorona(d)}, we conclude that $\dimension_A(P_4\odot  P_5)=n\cdot \dimension_A(P_5)+\gamma'(P_4)=4\cdot 2+1=9.$}
\end{figure}

As indicated in Figure~\ref{fig-P5},  $H=P_5$ satisfies the premises of Theorem \ref{TheOnlyPosibilitiesDimAdjCorona(d)},
as in particular there are adjacency bases that are also dominating set (see the leftmost copy of a $P_5$ in Figure~\ref{fig-P5})
as well as adjacency bases that are not dominating sets
(see the rightmost copy of a $P_5$ in that drawing).
Hence, we can conclude:

\begin{corollary}
For any connected graph $G$ of order $n\ge 2$, 
$$\dimension_A(G\odot P_5)=2n+\gamma'(G).$$
\end{corollary}


Since the assumptions of Theorems \ref{TheOnlyPosibilitiesDimAdjCorona(a)} \ref{TheOnlyPosibilitiesDimAdjCorona(b)}, \ref{TheOnlyPosibilitiesDimAdjCorona(c)} and \ref{TheOnlyPosibilitiesDimAdjCorona(d)} are complementary, we obtain the following result.  

\begin{remark} 
For any  connected graph $G$    of order $n\ge 2$ and any   non-trivial graph $H$,
 $$\dimension_A(G\odot H)=n\cdot \dimension_A(H)$$
or   $$\dimension_A(G\odot H)=n\cdot \dimension_A(H)+\gamma (G) $$
or   $$\dimension_A(G\odot H)=n\cdot \dimension_A(H)+\gamma' (G)$$
or $$\dimension_A(G\odot H)=n\cdot \dimension_A(H)+n-1.$$
\end{remark}

Moreover, since the assumptions of Theorems \ref{TheOnlyPosibilitiesDimAdjCorona(a)} \ref{TheOnlyPosibilitiesDimAdjCorona(b)}, \ref{TheOnlyPosibilitiesDimAdjCorona(c)} and \ref{TheOnlyPosibilitiesDimAdjCorona(d)} are complementary and  for any  graph $G$ of order $n\ge 3$ it holds that $0<\gamma'(G)\le \gamma(G)\le \frac{n}{2}<n-1$, we can conclude that  in fact,  Theorems \ref{TheOnlyPosibilitiesDimAdjCorona(a)} and \ref{TheOnlyPosibilitiesDimAdjCorona(d)} are equivalences for $n\ge 3$. 
Notice that for $n=2$, Theorem~\ref{TheOnlyPosibilitiesDimAdjCorona(a)} is also an equivalence. Therefore,  we obtain the following  two results.

\begin{theorem}\label{ThAdjacencyBasisDominatingSets}
Let $G$ be a  connected graph of order $n\ge 2$ and let $H$ be a non-trivial graph. The following statements are equivalent:
\begin{enumerate}[{\rm (i)}]
\item There exists an adjacency basis $S$ for $H$, which is also a dominating set, such that for every $v\in V(H)-S$ it is satisfied that $S\not\subseteq N_H(v)$.
\item $\dimension_A(G\odot H)=n\cdot \dimension_A(H).$

\item $\dimension_A(G\odot H)=\dimension(G\odot H).$
\end{enumerate}
\end{theorem}

As an example of application of Theorem \ref{ThAdjacencyBasisDominatingSets} we can take $H$ as the cycle graphs $C_r$ or the path graphs $P_r$, where  $r\ge 7$, 
 $r\not \equiv 1 \bmod 5$ and  $r\not\equiv 3 \bmod 5$, as explained in Cor.~\ref{Cor-TheOnlyPosibilitiesDimAdjCorona(a)}.

\begin{theorem}\label{noAdjBasisDominating}
Let $G$ be a  connected graph of order $n\ge 3$ and let $H$ be a non-trivial graph. The following statements are equivalent:
\begin{enumerate}[{\rm (i)}]
\item No adjacency basis  for $H$ is a dominating set.
\item $\dimension_A(G\odot H)=n\cdot \dimension_A(H)+n-1 .$

\item $\dimension_A(G\odot H)=\dimension(G\odot H)+n-1.$
\end{enumerate}
\end{theorem}

An example of graph $H$ where we can apply Theorem \ref{noAdjBasisDominating} is the star graph $K_{1,r}$ (see Cor.~\ref{Cor-TheOnlyPosibilitiesDimAdjCorona(c)}), $r\ge 2$, 
or the path graphs $P_r$, where  $r\ge 7$,  $r  \equiv 1 \bmod 5$ or  $r \equiv 3 \bmod 5$.

\section{The local metric dimension of corona product graphs versus  the local adjacency dimension of a graph}

In the beginning, 
we consider some straightforward cases.  If  $H$ is an empty graph, then $K_1\odot H$ is a star graph and $\dimension_l(K_1\odot H)=1.$ Moreover, if $H$ is a complete graph of order $n$, then $K_1\odot H $ is a complete graph of order $n+1$ and $\dimension_l(K_1\odot H)=n.$  It was shown in \cite{Rodriguez-Velazquez2013LDimCorona}
that for any  connected nontrivial graph $G$ and any empty graph $H$, $$\dimension_l(G\odot H)=\dimension_l(G).$$

As this section is organized similar to the previous one, we refrain from structuring it by explicit subsections.

Our next result allow us to express $\dimension_l(G\odot H)$ in terms of the order of $G$ and $\dimension_{A,l}(H)$.

\begin{theorem} \label{mainLocalTheoremDim}
For any connected graph $G$ of order $n\ge 2$ and any non-trivial graph $H$, 
$$\dimension_l(G\odot H)=n \cdot \dimension_{A,l}(H).$$
\end{theorem}

\begin{proof}
The results is deduced by analogy to the proof of Theorem \ref{mainTheoremDim} where the analysis is restricted to pairs of adjacent vertices.    
\end{proof}

Now we  point out some results obtained in \cite{Rodriguez-Velazquez2013LDimCorona}. 

\begin{theorem}{\rm \cite{Rodriguez-Velazquez2013LDimCorona}} \label{SurveyTheoremLocal}
Let $H$ be a non-empty graph of order $n'$ and let $G$ be a connected graph  of order $n\ge 2$. The following assertions hold.
\begin{enumerate}[{\rm (1)}]
\item  If the vertex of $K_1$ does not belong to any local metric basis for $K_1+H$, then for any connected graph $G$ of order $n$,
$$\dimension_l(G\odot H)=n\cdot \dimension_l(K_1+H).$$

\item   If the vertex of $K_1$  belongs to a local metric basis for $K_1+H$, then for any connected graph $G$ of order $n\ge 2$,
$$\dimension_l(G\odot H)=n  (\dimension_l(K_1+H)-1).$$

\item Let $t\ge 4$ be an integer. If $t\equiv 1 \bmod 4$, then $\dimension_l(G\odot P_t)=n \left\lfloor\frac{t}{4}\right\rfloor$, 
and if   $t\not \equiv 1 \bmod 4$, then $\dimension_l(G\odot P_t)=n \left\lceil\frac{t}{4}\right\rceil.$
\item For any integer $t\ge 4$, $\dimension_l(G\odot C_t)=n \left\lceil\frac{t}{4}\right\rceil$.

\item If  $H$ has diameter two, then 
$$\dimension_l(G\odot H)=n\cdot \dimension_l(H).$$

\item If $H$ has radius $r(H)\ge 4$, then 
 $$ \dimension_l(G\odot H)= n\cdot \dimension_l(K_1+H).$$

\item $\dimension_l(G\odot H)=n$ if and only if $H$ is a bipartite graph having only one non-trivial connected component $H^*$ and $r(H^*)\le 2$.

 \item $\dimension_l(G\odot H)=n(n'-1)$ if and only if $H\cong K_{n'}$ or $H\cong K_1\cup  K_{n'-1}$.
\end{enumerate}
\end{theorem}


According to the results  above,  we can conclude  the following theorem.
\begin{theorem} 
Let $H$ be a non-empty graph of order $n'$. The following assertions hold.
\begin{enumerate}[{\rm (1)}]
\item  If the vertex of $K_1$ does not belong to any local metric basis for $K_1+H$, then 
$$\dimension_{A,l}( H)=  \dimension_l(K_1+H).$$

\item   If the vertex of $K_1$  belongs to a local metric basis for $K_1+H$, then 
$$\dimension_{A,l}( H)= \dimension_l(K_1+H)-1.$$

\item Let $t\ge 4$ be an integer. If $t\equiv 1 \bmod 4$, then $\dimension_{A,l}(  P_t)=  \left\lfloor\frac{t}{4}\right\rfloor$,
and if   $t\not \equiv 1 \bmod 4$, then $\dimension_{A,l}(  P_t)=  \left\lceil\frac{t}{4}\right\rceil.$
\item For any integer $t\ge 4$, $\dimension_{A,l}(  C_t)= \left\lceil\frac{t}{4}\right\rceil$.

\item If  $H$ has diameter two, then 
$$\dimension_{A,l}(  H)=  \dimension_l(H).$$

\item If $H$ has radius $r(H)\ge 4$, then 
 $$ \dimension_{A,l}(  H)=   \dimension_l(K_1+H).$$

\item 
$\dimension_{A,l}(H)=1$ if and only if $H$ is a bipartite graph having only one non-trivial connected component $H^*$ and $r(H^*)\le 2$.
 
 \item $\dimension_{A,l}( H)= n'-1$ if and only if $H\cong K_{n'}$ or $H\cong K_1\cup  K_{n'-1}$.
\end{enumerate}
\end{theorem}


Fortunately, the comparison of the local adjacency dimension of the corona product with the one of the second argument
is much simpler in the local version as in the previously studied non-local version.

\begin{theorem}\label{TheOnlyPosibilitiesLocalDimAdjCorona(a)}
Let $G$ be a  connected graph of order $n\ge 2$ and let $H$ be a non-trivial graph.  
If there exists a local  adjacency basis $S$ for $H$ such that  for every $v\in V(H)-S$ it is satisfied that 
$S\not\subseteq N_H(v)$, then   $$\dimension_{A,l}(G\odot H)=n\cdot \dimension_{A,l}(H).$$
\end{theorem}

\begin{proof} 
Suppose that $S$ is a local adjacency basis for $H$.  Let $S_i$ be the copy of $S$ in the $i^{th}$ copy of $H$ in $G\odot H$. First of all, note that by Theorems \ref{mainLocalTheoremDim} we have
$$\dimension_{A,l}(G\odot H)\ge \dimension_l(G\odot H)=n\cdot \dimension_{A,l}(H).$$

Suppose that for every $v\in V(H)-S$ it is satisfied that $S\not\subseteq N_H(v)$.
We claim that $\dimension_{A,l}(G\odot H)\le n\vert S\vert$. To see this, let $S'=\bigcup_{i=1}^nS_i$ and let us prove that $S'$ is a local adjacency generator for $G\odot H$. So we differentiate the following cases for any pair $x,y$ of adjacent vertices of $G\odot H$  not belonging to $S'$. 

\begin{enumerate}
 \item  $x,y\in V_i$. Since $S_i$ is a local adjacency basis of $H_i$,  there exists some $u_i\in S_i$ such that either $u_i\sim x$ and $u_i \not\sim y$ or $u_i\not\sim x$ and $u_i \sim y$. 

\item  $x\in V_i$, $y=v_i\in V$. By assumption, we have that $S_i\not\subseteq N_{H_i}(x)$, so for every $u\in S_i-N_{H_i}(x)$,  we have $u\sim y$.

\item  $x=v_i,y=v_j\in V$, $i\ne j$. Taking $u\in S_i$, we have $u\sim x $ and $u\not \sim y$.

\end{enumerate}

From the  cases above, we conclude that $S'$ is a local adjacency generator for $G\odot H$ and, as a consequence,  $\dimension_{A,l}(G\odot H)\le \vert S' \vert = n\vert S\vert=n\cdot \dimension_{A,l}(H)$. The proof  is complete.  
\end{proof}


\begin{theorem}\label{TheOnlyPosibilitiesDimLocalAdjCorona(b)}
Let $G$ be a  connected graph of order $n\ge 2$ and let $H$ be a non-trivial graph.  
  If  for any local adjacency basis for $H$, there exists  some $v\in V(H)-S$ which satisfies  that $S\subseteq N_H(v)$, 
  then   $$\dimension_{A,l}(G\odot H)=n\cdot \dimension_{A,l}(H)+\gamma (G).$$
\end{theorem}

\begin{proof}  
Let $W$ be a local adjacency basis for $G\odot H$ and let $W_i=W\cap V_i$ and $U=W\cap V$. Since two adjacent vertices belonging to $V_i$ are not distinguished by any $u\in W-V_i$, the set $W_i$ must be  a local adjacency generator for $H_i$. Now consider the partition $\{V',V''\}$ of $ V$   defined as follows:
$$V'=\{v_i\in V: \; \; \; \vert W_i\vert = \dimension_{A,l}(H)\}\; \; {\rm  and }\;\;
V''=\{v_j\in V: \; \; \;\vert W_j\vert \ge  \dimension_{A,l}(H)+1\}.$$

 Note that, if $v_i\in V'$, then  $W_i$ is a local adjacency basis for $H_i$, thus in this case 
 there exists $u_i\in V_i$ such that $W_i\subset N_{H_i}(u_i)$. Then  the pair $u_i,v_i$ is not distinguished by the elements of $W_i$ and, as a consequence, either $v_i\in U$ or there exists $v_j\in U$ such that $v_j\sim v_i.$  Hence, 
 $U\cup V''$ must be a dominating set and, as a result, $$\vert U\cup V''\vert \ge \gamma(G).$$ 
 So we obtain the following:
\begin{eqnarray*}
\dimension_{A,l}(G\odot H)&=&\vert W\vert \\
             &= & \bigcup_{v_i\in V'} \vert W_i \vert + \bigcup_{v_j\in V''} \vert W_j\vert +\left\vert  U  \right\vert \\
             &\ge  & \sum_{v_i\in V'} \dimension_{A,l}(H) + \sum_{v_j\in V''} (\dimension_{A,l}(H)+1) +\left\vert  U  \right\vert \\
             & =  & n\cdot  \dimension_{A,l}(H) + \vert V''\vert +\left\vert  U  \right\vert \\
             &\ge & n\cdot  \dimension_{A,l}(H) +\vert  V'' \cup  U  \vert \\   
              & \ge & n\cdot \dimension_{A,l}(H)+\gamma(G).
\end{eqnarray*}

To conclude the proof,  we consider    a local  adjacency basis $S$ for $H$ and we denote by  $S_i$   
the copy of $S$ corresponding to $H_i$.  We  claim that for any dominating set $D$ of $G$ of minimum cardinality
$\vert D\vert=\gamma(G)$,  the set $D\cup (\bigcup_{i=1}^n S_i)$ is a local adjacency generator for $G\odot H$ and, as a result,
$$\dimension_{A,l}(G\odot H)\le \left\vert D\cup \left(\bigcup_{i=1}^n S_i \right)\right \vert =n\cdot \dimension_{A,l}(H)+\gamma(G).$$
To see this, we differentiate the same cases as in the proof of Theorem \ref{TheOnlyPosibilitiesLocalDimAdjCorona(a)}
with the  difference that now in Cases 2 either  $v_i\in D$ or $v_i$ is dominated by some element  of $D$  and,  
analogously, in Case 3 either $y=v_i\in D$ or there exists some $v_l\in D$ such that $v_l\sim y$. 
Of course, if $y=v_i\not \in D$, then    $v_l$ distinguishes the pair $x,y$. Therefore, the result follows. 
\end{proof}

\begin{remark}\label{remark-local-adjacency-hardness}
As a concrete example for the previous theorem, consider $H=K_{n'}$. 
Clearly, $\dimension_{A,l}(H)=n'-1$, and the neighborhood of the only vertex that is not in the local adjacency basis 
coincides with the local  adjacency basis. 
For any connected graph $G$ of order $n\geq 2$, we can deduce that 
  $$\dimension_{A,l}(G\odot K_{n'})=n\cdot \dimension_{A,l}(K_{n'})+\gamma (G)=n(n'-1)+\gamma (G).$$
\end{remark}

Since the assumptions of Theorems \ref{TheOnlyPosibilitiesLocalDimAdjCorona(a)} and 
\ref{TheOnlyPosibilitiesDimLocalAdjCorona(b)} are complementary, we obtain the following property for $\dimension_{A,l}(G\odot H)$.

\begin{theorem}[Dichotomy]
For any connected graph $G$   of order $n\ge 2$ and any non-trivial graph $H$ either  
 $$\dimension_{A,l}(G\odot H)=n\cdot \dimension_{A,l}(H)$$
  or  
   $$\dimension_{A,l}(G\odot H)=n\cdot \dimension_{A,l}(H)+\gamma (G).$$
\end{theorem} 
Now, since for any graph $H$ it is satisfied that $0<\gamma  (H)$ and  the assumptions of Theorems \ref{TheOnlyPosibilitiesLocalDimAdjCorona(a)} and 
\ref{TheOnlyPosibilitiesDimLocalAdjCorona(b)} are complementary, we conclude that, in fact, Theorems  \ref{TheOnlyPosibilitiesLocalDimAdjCorona(a)} and 
\ref{TheOnlyPosibilitiesDimLocalAdjCorona(b)} are equivalences. 

\begin{theorem}
Let $G$ be a  connected graph of order $n\ge 2$ and let $H$ be a non-trivial graph.  Then the following assertions are equivalent. 

\begin{enumerate}[{\em (i)}]
\item There exists a local  adjacency basis $S$ for $H$ such that  for every $v\in V(H)-S$ it is satisfied that $S\not\subseteq N_H(v)$.

\item $\dimension_{A,l}(G\odot H)=n\cdot \dimension_{A,l}(H).$
\item $\dimension_l(G\odot H)=\dimension_{A,l}(G\odot H)$. 
\end{enumerate}
\end{theorem}

An example of graph $H$ where we can apply the above result is  the path $P_r$, $r\ge 6$. In this case, for any connected graph $G$ of order $n\ge 2$, the following is true:
$$\dimension_{A,l}(G\odot P_r)=
\left\lbrace
\begin{array}{ll}
 n \left\lfloor\frac{r}{4}\right\rfloor\; \mbox{\rm if } r\equiv 1 \bmod 4;
\\
\\
n\left\lceil\frac{r}{4}\right\rceil \; \mbox{\rm if } r\not\equiv 1 \bmod 4.
\end{array}
\right.
$$

\begin{theorem}
Let $G$ be a  connected graph of order $n\ge 2$ and let $H$ be a non-trivial graph.  Then the following assertions are equivalent. 

\begin{enumerate}[{\em (i)}]
\item For any local adjacency basis $S$ for $H$, there exists some $v\in V(H)-S$ which satisfies  that $S\subseteq N_H(v)$.

\item $\dimension_{A,l}(G\odot H)=n\cdot \dimension_{A,l}(H)+\gamma(G).$
\item $\dimension_l(G\odot H)=\dimension_{A,l}(G\odot H)-\gamma(G)$. 
\end{enumerate}
\end{theorem}

As a concrete example of graph $H$ where we can apply the above result is  the star $K_{1,r}$, $r\ge 2$. In this case, for any connected graph $G$ of order $n\ge 2$, we find that
$$\dimension_{A,l}(G\odot K_{1,r})=n\cdot \dimension_{A,l}(K_{1,r})+\gamma(G)=n+\gamma(G).$$

\section{Twins and strong products of graphs}

We define the \textit{twin  equivalence relation} ${\cal R}$ on $V(G)$ as follows:
$$x {\cal R} y \longleftrightarrow  N_G[x]=N_G[y] \; \; \mbox{\rm or } \; N_G(x)=N_G(y).$$

We have three possibilities for each twin equivalence class $U$:

\begin{enumerate}[(a)]
\item $U$ is a singleton set, or
\item   $N_G(x)=N_G(y)$, for any $x,y\in U$ (and case (a) does not apply), or
\item   $N_G[x]=N_G[y]$, for any $x,y\in U$ (and case (a) does not apply).
\end{enumerate}

We will  refer to the type (c) classes as the \textit{true twin  equivalence classes}, \textit{i.e.,} $U$ is a true twin  equivalence class if and only if $|U|>1$ 
 and $N_G[x]=N_G[y]$, for any $x,y\in U$. 

Let us see three different examples where every vertex is twin. An example of a graph where every equivalence class is a true twin equivalence class is 
$K_r+(K_s\cup K_t)$, $r,s,t\ge 2$. In this case, there are three equivalence
classes composed of $r,s$ and $t$  true twin vertices, respectively. 
As an example where no class is composed of true twin vertices, we take the complete bipartite graph $K_{r,s}$, $r,s\ge 2$. 
Finally, the graph $K_r+N_s$,  $r,s\ge 2$, has two equivalence classes and one of them is composed of $r$ true twin vertices. 
On the other hand, $K_1+(K_r\cup N_s)$, $r,s\ge 2$, is an example where one class is singleton, one class is composed of true twin vertices and the other one is composed of false twin vertices.

If $U$ is a twin equivalence class in a connected graph G with $\vert U\vert = r \ge 2$, then every metric generator for $G$ contains at least $r-1$ elements from $U$. Thus, we point out  the following remark stated in \cite{Chartrand2003}.

\begin{remark}{\rm \cite{Chartrand2003}}\label{observationDimClasses}
Let $G$ be a connected graph of order $n$.  If $G$ has $t$  twin equivalence classes, then
$$\dimension (G)\ge n-t.$$ 
\end{remark}

\begin{theorem}\label{mainTheoremClasses}
Let $G$ be a connected graph of order $n$ having $t$  twin equivalence classes.  If  $G$ does not have singleton twin equivalence classes, then $$\dimension_A(G)=\dimension (G)=n-t.$$ 
\end{theorem}

\begin{proof}
Let  $U_1,U_2,\ldots,U_t$ be the twin equivalence classes of $G$. Let $u_i$ be an arbitrary element of $U_i$, for each $i\in \{1,\ldots,t\}$. 
We claim that  $W:=\bigcup_{i=1}^t(U_i-\{u_i\})$ is an adjacency generator for $G$. 
To see this, we differentiate three cases for the pairs of twin equivalence classes $U_i,U_j$, $i\ne j$.
\begin{enumerate}
 \item 
$U_i$ and $U_j$ are composed of true twin vertices. If $u_i\not\sim u_j$, then for every $u\in U_i-\{u_i\}$ we have $u\sim u_i$ and $u\not \sim u_j$. 
Now we suppose that $u_i\sim u_j$. Since $N_G[u_i]\ne N_G[u_j]$, there exists $u\in U_l-\{u_l\}$, $l\ne i,j$, such that 
either $u\sim u_i$ and $u\not \sim u_j$ or $u\sim u_j$ and $u\not \sim u_i$. Hence, $u_i$ and $u_j$ are distinguished by $u\in W$.
\item 
$U_i$ and $U_j$ are composed of false twin vertices. If $u_i\sim u_j$, then for every $u\in U_i-\{u_i\}$ we have $u\not \sim u_i$ and $u \sim u_j$. 
Now we assume that $u_i\not\sim u_j$.
Since $N_G(u_i)\ne N_G(u_j)$, there exists $u\in U_l-\{u_l\}$, $l\ne i,j$, such that 
either $u\sim u_i$ and $u\not \sim u_j$ or $u\sim u_j$ and $u\not \sim u_i$. Hence, $u_i$ and $u_j$ are distinguished by $u\in W$.

\item 
$U_i$ is composed of true twin vertices and  $U_j$ is composed of false twin vertices. 
If $u_i\not\sim u_j$, then for every $u\in U_i-\{u_i\}$ we have $u\sim u_i$ and $u\not \sim u_j$. 
Similarly, if $u_i\sim u_j$, then for every $u\in U_j-\{u_j\}$ we have $u\sim u_i$ and $u\not \sim u_j$. So, $u_i$ and $u_j$ are distinguished by $u\in W$.

\end{enumerate}
We conclude that $W$ is an adjacency generator for $G$ and  that $\dimension(G)\le \dimension_A(G)\le \vert W\vert =n-t.$ 
Moreover, by Remark~\ref{observationDimClasses} we have $\dimension_A(G)\ge \dimension (G)\ge n-t.$ Therefore, the proof is complete.
\end{proof}

\begin{lemma}\label{classesINstrongProduct}
Let $G$ and $H$ be two connected graphs. If $G$ and $H$ have $t$ and $t'$ true twin  equivalence classes, 
and they have $n_1$ and $n_1'$ vertices not belonging to any true twin  equivalence class, respectively,   
then $G\boxtimes H$ has  $n_1t'+n'_1t+tt'$  true twin equivalence classes and the remaining twin equivalence  classes $($if any$)$ are singleton. 
\end{lemma}

\begin{proof}
Let $U_1,\ldots,U_t$ and $U'_1,\ldots,U'_{t'}$ be the true twin equivalence classes of $G$ and $H$, respectively. 

For any two vertices $a,c\in U_i$ and $b\in V(H)$,
\begin{align*}N_{G\boxtimes H}[(a,b)]&= \{(x,y): \; x\in N_G[a], y\in N_H[b]\}\\
&= \{(x,y): \; x\in N_G[c],y\in N_H[b]\}\\
&=N_{G\boxtimes H}[(c,b)].
\end{align*}
Thus, $(a,b)$ and $(c,b)$ are true twin vertices.
By analogy we check that  for any two vertices $b,c\in U'_j$ and $a\in V(G)$, it follows that $(a,b)$ and $(a,c)$ are true twin vertices.

Then we have that the sets of the form $U_i\times U_j$ are composed of true twin vertices. To conclude that $U_i\times U_j$ is a true twin equivalence class,
 we take $(a,b)\in U_i\times U_j$ and we  differentiate two cases for any $(x,y)\not\in U_i\times U_j$.
\begin{enumerate}
 \item 
$x\not\in U_i$. Since $x$ and $a$ are not true twin in $G$,   
either there exists $a_x\in N_G(a)-N_G[x]$ or there exists $x_a\in N_G(x)-N_G[a]$. 
Thus, we have two possibilities in $G\boxtimes H$:  either $(a,b)\sim (a_x,b)\not \sim(x,y)$ or  $(x,y)\sim (x_a,y)\not \sim(a,b)$. 

\item 
$y\not\in U'_i$. Now we proceed by analogy to Case 1. 
Since $y$ and $b$ are not true twin in $H$,   
either there exists $b_y\in N_H(b)-N_H[y]$ or there exists $y_b\in N_H(y)-N_H[b]$. 
Thus, we have two possibilities in $G\boxtimes H$:  either $(a,b)\sim (a,b_y)\not \sim(x,y)$ or  $(x,y)\sim (x,b_y)\not \sim(a,b)$. 
\end{enumerate}
In both cases, Case 1 and Case 2, $(a,b)$ and $(x,y)$ are not true twin vertices in $G\boxtimes H$ and, as a result, $U_i\times U_j$ is a true twin equivalence class in $G\boxtimes H$. 

By a similar process  we conclude that for every $a,x\in V(G)-\bigcup_{i=1}^tU_i$ and every $b,y\in V(H)-\bigcup_{i=1}^{t'}U'_i$ 
the sets of the form  $\{a\}\times U'_i$ or $U_i\times \{b\}$ are true twin equivalence classes and the vertices of the form $(a,b), (x,y)$ are neither true twins nor false twins in $G\boxtimes H$. 
\end{proof}

Note that according to the  lemma above, if $G$ and $H$ are connected bipartite graphs different from $K_2$, then all the twin equivalence classes of $G\boxtimes H$ are singleton sets. 
A more general result is stated by the next corollary. 

\begin{corollary}\label{StrongAllSingleton}
Let $G$ and $H$ be two connected graphs. If $G$ and $H$ do not have true twin vertices, then  all the twin equivalence classes of $G\boxtimes H$ are singleton sets. 
\end{corollary}

Another interesting consequence of Lemma is that for any connected bipartite graph $H$ of order $n'$, and any integer $n\ge 2$, $V(K_n\boxtimes H)$ is partitioned into $n'$ true twin classes. 
This fact is generalized by the following result.  

\begin{corollary}
  Let $G$ and $H$ be two connected graphs.  
If  $V(G)$ is partitioned into $t$ true twin equivalence classes and $H$ does not have true twin vertices, then $V(G\boxtimes H)$ is partitioned into $t n'$ true twin classes.  
\end{corollary}

Now we would point out some  direct consequence of combining Remark \ref{observationDimClasses} and Theorem \ref{mainTheoremClasses}  with Lemma \ref{classesINstrongProduct} and its consequences. 

\begin{theorem}
Let $G$ and $H$ be two connected graphs of order $n$ and $n'$, respectively. 
If $G$ and $H$ have $t$ and $t'$ true twin  equivalence classes, and $n_1$ and $n_1'$ vertices not belonging to any true twin  equivalence class, respectively,   then 
$$\dimension (G\boxtimes H)\ge  nn'-n_1t'-n'_1t-tt'-n_1n_1'.$$ 
Moreover, if $V(G)$ is partitioned into $t$ true twin equivalence classes,  then 
$$\dimension _A(G\boxtimes H) =  \dimension (G\boxtimes H)= nn'-n'_1t-tt'.$$
\end{theorem}

\begin{corollary}\label{BipartiteTimesPartitionable}
Let $G$ and $H$ be two connected graphs of order $n$ and $n'$, respectively.
If $V(G)$ is partitioned into $t$ true twin equivalence classes and $H$ does not have true twin vertices,
 then $$\dimension_A(G\boxtimes H) =  \dimension(G\boxtimes H)= n'(n-t).$$
\end{corollary}

Given a family $H_1,H_2,...,H_k$ of graphs we denote 
$$\prod_{i=1}^k{_{\boxtimes}} H_i=H_1\boxtimes H_2\boxtimes \cdots \boxtimes H_k.$$

We emphasize the following particular case of Corollary \ref{BipartiteTimesPartitionable}, which is also derived from Theorem \ref{mainTheoremClasses} and Corollary  \ref{StrongAllSingleton}.

\begin{remark}Let $n\ge 2$ be an integer. For any family $H_1,H_2,\ldots,H_k$ of bipartite graphs of order $n_1, n_2, \ldots,n_k$, respectively,
$$\dimension_A\left(K_n\boxtimes \left(\prod_{i=1}^k{_{\boxtimes}} H_i\right)\right)=dim\left(K_n\boxtimes \left(\prod_{i=1}^k{_{\boxtimes}} H_i\right)\right)=(n-1)\prod_{i=1}^k n_i.$$
\end{remark}

\section{The computational complexity of the four  dimension variants}

In this section, we not only prove NP-hardness of all dimension variants, but also show that the problems  (viewed as minimization problems)  cannot be solved in time
$O(pol(n+m)2^{o(n)})$   on any graph of order $n$ (and size $m$).
Yet, it is straightforward to see that each of our computational problems can be solved in time $O(pol(n+m)2^n)$,
simply by cycling through all vertex subsets by increasing cardinality and then checking if the considered vertex set forms an appropriate basis.
More specifically, based on our reductions we can conclude that these trivial brute-force algorithms are in a sense optimal, assuming the validity of the Exponential Time Hypothesis (ETH).
A direct consequence of ETH (using the sparsification lemma) is the hypothesis that 3-SAT instances cannot be solved in time $O(pol(n+m)2^{o(n+m)})$ 
on instances with $n$ variables and $m$ clauses, see \cite{ImpPatZan2001,CalImpPat2009}.

From a mathematical point of view, the most interesting fact is that most of our computational results are based on the combinatorial results
on the dimensional graph parameters on corona and strong products of graphs that are derived 
earlier in this paper.

Due to the practical motivation of the parameters, we also study their computational complexity on planar graph instances.

We are going to study the following  problems:

\noindent\textsc{ADim}: Given a graph $G$ and an integer $k$, decide if $\dimension(G)\leq k$ or not.\\
\textsc{LocDim}: Given a graph $G$ and an integer $k$, decide if $\dimension_{l}(G)\leq k$ or not.\\
\noindent\textsc{AdjDim}: Given a graph $G$ and an integer $k$, decide if $\dimension_A(G)\leq k$ or not.\\
\textsc{LocAdjDim}: Given a graph $G$ and an integer $k$, decide if $\dimension_{A,l}(G)\leq k$ or not.

As auxiliary problems, we will also consider:\\
\textsc{VC}:  Given a graph $G$ and an integer $k$, decide if $vc(G)\leq k$ or not.\\
\textsc{Dom}: Given a graph $G$ and an integer $k$, decide if $\gamma(G)\leq k$ or not.\\
\textsc{1-LocDom}: Given a graph $G$ and an integer $k$, decide if there exists a 1-locating dominating set of $G$ with at most $k$ vertices or not.
%
Recall that a dominating set $D\subseteq V$ in a graph $G=(V,E)$ is called a 1-locating dominating set if for every two vertices $u,v\in V\setminus D$,
the symmetric difference of 
$N(u)\cap D$ and 
$N(v)\cap D$ is non-empty.

We first recall the following result first mentioned in the textbook of Garey and Johnson~\cite{Garey1979}, with a proof first published in~\cite{Khuller1996}.

\begin{theorem}\label{Dim-complete}
\textsc{Dim} is NP-complete, even when restricted to planar graphs.
\end{theorem}

\begin{remark}\label{rem-mainTheoremDim}
Different proofs of this type of hardness result appeared in the literature.
For planar instances, we refer to \cite{DiaPSL2012}. 
In fact, we can offer a further one, based upon Theorem~\ref{mainTheoremDim} and the following result.
Namely, if there were a polynomial-time algorithm for computing $\dimension(G)$, then we could
compute $\dimension_A(H)$ for any graph (non-trivial) $H$ by computing $\dimension(K_2\odot H)$ with the assumed polynomial-time algorithm,
knowing that this is just twice as much as $\dimension_A(H)$.
As every NP-hardness proof adds a bit to the understanding of the nature of the problem, this one does so, as well.
It shows that 
\textsc{Dim} is NP-complete even on the class of graphs that can be written as $G\odot H$, where $G$ is some connected graph of order $n\geq 2$ and $H$ is non-trivial.
\end{remark}

\begin{theorem}\label{AdjDim-complete}
\textsc{AdjDim} is NP-complete, even when restricted to planar graphs.
\end{theorem}

\begin{proof}
 Membership in NP is easy to see.
We reduce from \textsc{1-LocDom}, see \cite{ChaHudLob2003,ColSlaSte87} for the NP-hardness, and also Theorem~\ref{thm-LocDom-hardness} below.
%
 Clearly, any 1-locating dominating set is also an adjacency generator, but the converse need not be true, as an adjacency generator need not be a dominating set.
However, if an adjacency generator is not a dominating set, then there is exactly one vertex which is not dominated.
Hence, we propose the following reduction: From an instance $G=(V,E)$ and $k$ of  \textsc{1-LocDom}, produce an instance $(G',k)$ of \textsc{AdjDim} by
obtaining $G'$ from $G$ by adding a new isolated vertex $x\notin V$ to $G$.
We claim that $G$ has a 1-locating dominating set of size at most $k$ if and only if $\dimension_A(G')\leq k$. 
Firstly, every  1-locating dominating set $D$ of $G$ is also an adjacency generator of $G'$, as $x$ is the only vertex that is not contained in $N[D]$ in $G'$.
Secondly, let $S$ be an adjacency generator of $G'$ of size at most $k$. If there is no vertex $v$ with $v\notin N[S]$, then $S\cap V$ is a 1-locating dominating set of size at most $k$ for $G$.
Otherwise, there is a  vertex $v$ with $v\notin N[S]$. If $v\in V$, then $x\in S$, as otherwise $x$ and $v$ cannot be differentiated. As $x\in S$ does not help distinguish any two
vertices $u,w\in V$, $S'=(S\setminus \{x\})\cup\{v\}$ is another adjacency generator for $G'$ of size at most $k$.
Hence, we can assume that for  an adjacency generator $S$  of $G'$ with  a  vertex $v$ with $v\notin N[S]$, $v\notin V$ holds, \textit{i.e.}, $v=x$. Then, $S$ is also a  1-locating dominating set.
\end{proof}

As we like to exploit further properties of the reduction, we provide a reduction for NP-hardness of \textsc{1-LocDom} in the following.
We need some further auxiliary results that might be interesting on their own.

%

\begin{lemma}\label{lem-VC-hardness}
 Assuming ETH, there is no $O(pol(n+m)2^{o(n)})$ algorithm solving \textsc{VC} on graphs of order $n$ and size $m$.
\end{lemma}

\begin{proof}
 The textbook reduction~\cite{Garey1979} shows just this, as it produces, starting from a 3-SAT formula with $n$ variables and $m$ clauses, a
 graph of order $3m+2n$ and size $6m+n$. 
\end{proof}

\begin{lemma}\label{lem-Dom-hardness}
 Assuming ETH, there is no $O(pol(n+m)2^{o(n)})$ algorithm solving \textsc{Dom} on graphs of order $n$ and size $m$.
\end{lemma}

\begin{proof}
 The textbook reduction of \cite[Theorem 1.7]{Haynes1998}
takes a 3-SAT formula with $n$ variables and $m$ clauses
and produces a graph 
 of order $3n+m$ and size $3n+3m$.

An alternative well-known reduction works as follows: It takes a \textsc{VC} instance $G$ of order $n$ and size $m$  and produces a graph 
 of order  $n'$ and size $m'$
  by replacing any edge of $G$ by a triangle, so that $n'=n+m$ and $m'=3m$.
 Hence, the claim follows by the previous lemma. We will call the second construction \emph{triangle construction} in the following.
\end{proof}

%

Notice that the two proofs of the preceding lemmas preserve planarity. 
This means that if the clause-and-variable graph associated to a Boolean formula (as introduced by Lichtenstein in \cite{Lic82})
is planar, then the three graphs resulting from the construction sketched in the preceding two lemmas are also planar.
This is important to notice, as this fact will be used in the proof of the next theorem. 
Notice that the NP-hardness itself already follows from the statement given in \cite{ColSlaSte87}, but that proof (starting out again from 3-SAT) does not preserve planarity,
as the variable gadget alone already contains a $K_{2,3}$ subgraph that inhibits non-crossing interconnections with the clause gadgets.

\begin{theorem}
\label{thm-LocDom-hardness} \textsc{1-LocDom} is NP-hard, even when restricted to planar graphs. Moreover, 
 assuming ETH, there is no $O(pol(n+m)2^{o(n)})$ algorithm solving \textsc{1-LocDom} on general graphs of order $n$ and size $m$.
\end{theorem}

\begin{proof}
 Membership in NP is easy to see. We start our reduction with a 3-SAT instance (with $n$ variables and $m$ clauses).  
We can assume that each variable occurs at least once positively and at least once negatively in the corresponding 3-SAT formula.
Likewise, we can assume that no literal occurs twice in any clause.
 Firstly, recall the standard construction for showing NP-hardness of \textsc{Vertex Cover}, see Lemma~\ref{lem-VC-hardness}.
 For each variable, two vertices are introduced, and for each clause, three vertices (for the involved literals) are put into the graph.
 As in the following other vertices will be added to the graph, we will refer to the vertices introduced by the \textsc{Vertex Cover} reduction as
 literal vertices, both in the clause gadgets and in the variable gadgets.
 Then, perform the triangle construction as indicated in Lemma~\ref{lem-Dom-hardness}.
 We obtain a graph $G$ of order $(2n+3m)+(n+6m)=3n+9m$ and of size $3n+18m$.
 In that graph, the ``variable gadgets'' contain three vertices (two of them being literal vertices), while the 
 ``clause gadgets'' contain six vertices, with three literal vertices among them. 
 As each literal vertex contained in a clause gadget is connected by a triangle with the corresponding
 literal vertex in the variable gadget, $3m$ more ``interconnection vertices'' are contained in $G$.
 $G$ has a dominating set $D$ of size (at most) $n+2m$ if and only if the given 3-SAT instance is satisfiable.
 We can assume that none of the vertices added by the triangle construction is in $D$. 
 Moreover, exactly two vertices per clause gadget belong to the dominating set, and one vertex per variable gadget. 
 Now, we show that $D$ is also a locating set. 
 To this end, we study two neighbors $u,v$ of some vertex $x\in D$ that do not belong to $D$.
 We have to show that the symmetric difference of $N(u)\cap D$ and of $N(v)\cap D$ is not empty.

 \begin{itemize}
  \item If $x$ is in some variable gadget, then three subcases arise:
  \begin{itemize}
   \item $u,v$ are in  variable gadgets. Clearly, $u,v$ are in the same variable gadget as $x$ is.  $x\in D$ means that the corresponding literal is set to true. Without loss of generality, $u$ is the other literal variable,
   while $v$ is added by the triangle construction. As the literal to which $u$ corresponds also occurs in the original 3-SAT formula, there is some literal vertex $u'$ in
   some clause gadget that is a neighbor of $u$. By construction, $u'\in D$, but $u'\notin N(v)=\{x,u\}$. 
   \item $u,v$ are in  clause gadgets. As no literal occurs twice in any clause of the given formula, $u$ and $v$ correspond to two different clause gadgets. 
   Hence,  $v$ has some neighbor $v'$ in that clause gadget that belongs to $D$, but $v'\notin N(u)$.
   \item $u$ is in a variable gadget, $v$ is in a clause gadget. As $v$ has some neighbor $v'$ in that clause gadget that belongs to $D$, but $v'\notin N(u)$, the symmetric difference of $N(u)\cap D$ and of $N(v)\cap D$ is not empty.
  \end{itemize}
\item If $x$ is the some clause gadget, again three subcases arise:

  \begin{itemize}
   \item $u,v$ are in  variable gadgets.  By construction, $u$ and $v$ belong to different variable gadgets. In these different gadgets, different variable vertices belong to $D$
   that are neighbors of $u$ or $v$, respectively.
   \item $u,v$ are in  clause gadgets. By construction, they are in the same clause gadget. Two subcases may now occur:
   \begin{itemize}
    \item $u$ and $v$ have been added by the triangle construction. As $D$ must contain exactly one further literal vertex $y$ from the clause gadget (apart from $x$),
    $y$ must be neighbor of either $u$ or $v$, but it cannot be neighbor of both of them. Hence, $y$ is in the symmetric difference of $N(u)\cap D$ and  $N(v)\cap D$.
    \item  $u$ is a literal vertex, but $v$ has been added by the triangle construction. Then, the partner literal vertex $u'$ in the corresponding variable gadget must
    belong to $D$. By construction, it is not a neighbor of $v$.
   \end{itemize}
Notice that the imaginable third case (both $u$ and $v$ are literal vertices) cannot occur, as two out of the three literal vertices in a clause gadget
belong to $D$, so that not both $u$ and $v$ can belong to the complement of $D$ as required.
   \item $u$ is in a variable gadget, $v$ is in a clause gadget. Hence, $v$ is a literal vertex. 
   Then, $v$ must have another literal vertex $v'$ in the same clause gadget as $v$ that belongs to $D$.
   By construction, $u$ is not a neighbor of $v'$.
  \end{itemize}
 \end{itemize}
Finally, observe that also Lichtenstein's construction for showing NP-hardness of \textsc{Planar Vertex Cover} can be modified to show NP-hardness of \textsc{Planar 1-LocDom}.
The variable gadgets are no longer single edges (as in the classical \textsc{Vertex Cover} NP-hardness reduction) but cycles of length $2m$.
Applying the triangle construction allows for the same arguments as given above. This concludes the proof.
\end{proof}

By the proof of Theorem~\ref{AdjDim-complete}, we can now conclude:

\begin{corollary}
 Assuming ETH, there is no $O(pol(n+m)2^{o(n)})$ algorithm solving \textsc{AdjDim} on graphs of order $n$ and size $m$.
\end{corollary}

As explained in Remark~\ref{rem-mainTheoremDim},
Theorem~\ref{mainTheoremDim} can be used to deduce furthermore:

\begin{corollary}
 Assuming ETH, there is no $O(pol(n+m)2^{o(n)})$ algorithm solving \textsc{Dim} on graphs of order $n$ and size $m$.
\end{corollary}

From Remark~\ref{remark-local-adjacency-hardness} and  Lemma~\ref{lem-Dom-hardness},
we can conclude, as  membership in NP is easy to see:

\begin{theorem}\label{LocAdjDim-complete}\textsc{LocAdjDim} is NP-complete.
Moreover, assuming ETH, there is no $O(pol(n+m)2^{o(n)})$ algorithm solving \textsc{LocAdjDim} on graphs of order $n$ and size $m$.
\end{theorem}


We provide an alternative proof of the previous theorem in the appendix of this paper.
That proof is a direct reduction from 3-SAT and is, in fact, very similar to the textbook proof 
for the NP-hardness of \textsc{Vertex Cover}. This also proves that \textsc{LocAdjDim} is NP-complete when restricted to planar instances.

As explained in Remark~\ref{rem-mainTheoremDim}, we can (now) use Theorem~\ref{mainLocalTheoremDim} together with Theorem~\ref{LocAdjDim-complete}
to conclude the following hitherto unknown complexity result. (Membership in NP is again easy to see.)

\begin{theorem}\label{LocDim-complete}
\textsc{LocDim} is NP-complete. 
Moreover, assuming ETH, there is no $O(pol(n+m)2^{o(n)})$ algorithm solving \textsc{LocDim} on graphs of order $n$ and size $m$.
\end{theorem}

Notice that the reduction explained in Remark~\ref{rem-mainTheoremDim} does not help find any hardness results on planar graphs.
Hence, we leave it as an open question whether or not \textsc{LocDim} is NP-hard also on planar graph instances.

Furthermore, let us point to the fact that the twin equivalence classes are quite easy to compute.
Therefore, the formula shown in Theorem~\ref{mainTheoremClasses} allows us to conclude that 
singleton twin equivalence classes are essential for the NP-hardness results that we obtained in this section.

\section{Conclusions}

We have studied four dimension parameters in graphs.
In particular, establishing concise formulae for corona product graphs allowed to deduce NP-hardness results 
(and similar hardness claims) for all these graph parameters, based on known results, in particular on \textsc{Vertex Cover}
and on \textsc{Dominating Set} problems.
We hope that the idea of using (combinatorial) formulae for computational hardness proofs can be also applied in other situations.

Let us conclude with indicating some possible future research directions.

\begin{itemize}
 \item Given some metric $D$ on the vertex set of a (connected) graph $G=(V,E)$ and some vertex set $S\subseteq V$, one can define the following relation $\sim_{D,S}$ on $V$:
$$u \sim_{D,S} v\iff \forall x\in S: D(x,u)=D(x,v).$$
Clearly, for any $D,S$, $\sim_{D,S}$ is an equivalence relation on $V$. Moreover, $S$ is a metric generator (with respect to the metric $D$) if and only if 
all equivalence classes of $\sim_{D,S}$  are singleton sets. 
We can then introduce the \emph{$D$-dimension} as the size of the smallest  metric generator with respect to the metric $D$ on $G$.

\underline{Research Question:} So far, we focussed on the metrics $d_G$ and $d_{G,2}$.
One might also study other metrics, like $d_{G,k}(x,y)=\min\{d_G(x,y),k\}$ for $k>2$. 
This way, also other notions of metric bases can be investigated, as well as the according graph dimension parameters.
First studies might focus on combinatorial aspects.

\item From the point of view of the previous item, we can call a set $S$ a \emph{local metric generator} (with respect to the metric $D$ on $G$) if 
each equivalence class of $\sim_{D,S}$ forms an independent set. 
We can then introduce the \emph{local $D$-dimension} as the size of the smallest local metric generator with respect to the metric $D$ on $G$.

\underline{Research Question:} Generalizations as suggested in the previous item can also be undertaken for the local $D$-dimension.

\item So far, we only focussed on proving computational hardness results for the four graph dimension notions that we studied in this paper.
This is usually only the beginning of an algorithmic research line that deals with the following 
\underline{Research Questions:} 
\begin{enumerate}
 \item Describe the boundary between polynomial-time solvability  and NP-hardness in terms of graph classes. We already explicitly mentioned several results 
on computing the four graph dimension  parameters when restricted to planar graphs.

\item Investigate the approximability of the graph parameters, viewed as minimization problems.
\item Study aspects of parametrized complexity for these graph parameters.
\item Devise algorithms with running times like $O(pol(n,m)\cdot c^n)$ for the problems on graphs of order $n$ and of size $m$, with $c<2$.
(This is also motivated by the hardness results based on ETH as presented in this paper.)
\end{enumerate}

\item \underline{Research Question:} Study computational hardness questions for the problems related to the parameters suggested in the first two items.
Then, the research program sketched in the third item might trigger for such parameters, as well.

\end{itemize}

\section{Appendix}

\begin{theorem}\label{LocAdjDim-complete2}\textsc{LocAdjDim} is NP-complete, even when restricted to planar instances.
\end{theorem}

\begin{figure}
\centerline{\scalebox{.4}{\includegraphics{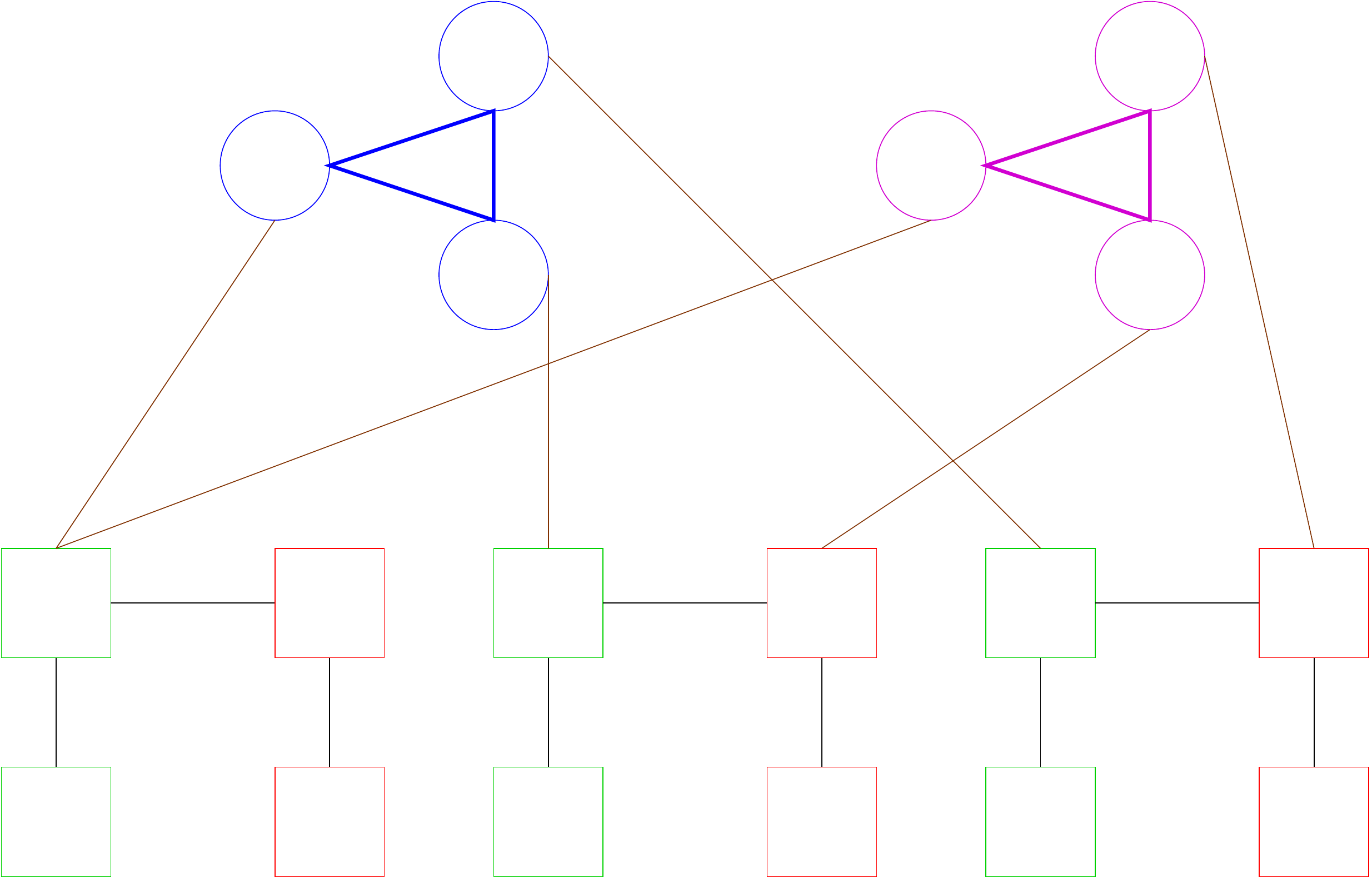}}}
\caption{\label{fig-LocAdjDim gagdet example}A small example illustrating the overall structure of the reduction for the clauses $(x\lor y\lor z)$ and $(x\lor \neg y\lor \neg z)$.  
There are three variables in the formula, $x,y,z$, in that order.
To each of them, an induced $P_4$ belongs whose vertices are coloured green and red.
The three green middle vertices of these paths are the positive literal vertices, while the three red middle vertices of these paths are the negative literal vertices.
The two triangle-shaped parts of the graph are gadgets for clauses. Here, they represent $(x\lor y\lor z)$ (blue) and $(x\lor \neg y\lor \neg z)$ (magenta).}
\end{figure}

\begin{proof}
 Membership in NP is easy to see.

For the hardness part, we propose a reduction from \textsc{3-SAT} that is similar to the standard textbook reduction for proving NP-hardness of \textsc{Vertex Cover}, confer, e.g., \cite{Garey1979}. 
An illustration of the construction is shown in Fig.~\ref{fig-LocAdjDim gagdet example}.

For each Boolean variable, we introduce four vertices that form a $P_4$. This path will be induced in the final graph, and connections to other graph parts would
be only possible via the two middle vertices. 
 The combinatorial claim is that finally exactly one of the two middle vertices should be in any local adjacency basis of the graph instance that 
we construct. In the following, these two middle vertices are called literal vertices, as whether or not they belong to the local adjacency generator determines whether or not the literal is
set to true. Due to this intention, we will call one of the  two middle vertices  positive literal vertex and the other one negative   literal vertex.

\begin{figure}
\centerline{
\scalebox{.4}{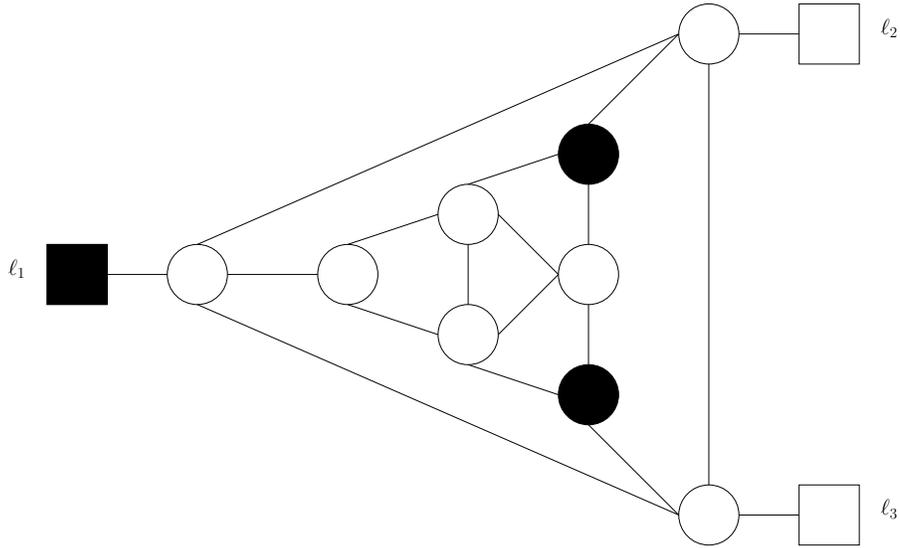}
}\caption{\label{fig-LocAdjDim clausegadget}The clause gadget illustration. The square-shaped vertices do not belong to the gadget, but they are the three literal vertices in variable gadgets
 that
correspond to the three literals in the clause.}
\end{figure}

We introduce one clause gadget per clause. This is a graph of order nine depicted in Fig.~\ref{fig-LocAdjDim clausegadget}. 
We claim that we need at least two vertices from each of these clause gadgets in any local adjacency basis. Two are only sufficient if some of the literal vertex neighbors from variable gadgets
are in the local adjacency basis. This can be seen in Fig.~\ref{fig-LocAdjDim clausegadget} by considering the vertices coloured black.
Also, we need at least two vertices in any local adjacency basis that are from the ``innermost'' six vertices in each gadget.
We assume that the three outermost vertices of each gadget are numbered like $1$, $2$, $3$. 

The overall structure of the graph $G=(V,E)$ belonging to some formula $F$ given by some set $X$ of $n$ variables and some set of $m$ 3-element clauses $C$ is as follows:
\begin{itemize}
 \item Introduce an induced $P_4$, called $p(x)$ for each variable $x\in X$. 
\item Introduce a subgraph $g(c)$ of order nine for each clause $c\in C$.
\item Assume that there is some order $<$ on $X$, which transfers to the set $X(c)$ of variables occurring in clause $c$. 
Hence, we can refer to the $i^{th}$ vertex in $X(c)$. 
\item An edge interconnects the positive literal vertex of $p(x)$ with the outermost vertex $o$ of $g(c)$ if and only if the literal $x$ occurs in $c$,
$x$ is the $i^{th}$ variable in $X(c)$
 and
$o$ is the vertex number $i$.
\item An edge interconnects the negative literal vertex of $p(x)$ with the outermost vertex $o$ of $g(c)$ if and only if the literal $\bar x$ occurs in $c$,
$x$ is the $i^{th}$ variable in $X(c)$
 and
$o$ is the vertex number $i$.
\item There are no further edges in the graph $G$.
\end{itemize}

Hence, $|V|=4n+9m$, $|E|=3n+18m$.

The overall claim is that there is a local adjacency basis of size at most (and also exactly) $2m+n$ if and only if the given 3-SAT formula  $F$ was satisfiable.
Our previous reasoning already explained that the local adjacency dimension of the union of the variable and clause gadget graphs is at least $2m+n$.
Furthermore, we claim that for any local adjacency basis that does not contain any vertex of some path $p(x)$, there exists another local adjacency basis 
that contains a middle vertex from $p(x)$.
Having a local  adjacency basis $A$ for $G$ with two vertices from each clause gadget and one middle vertex from each variable gadget, we obtain a satisfying assignment of
the  formula $F$ by setting $x$ to true if and only if the positive literal of $p(x)$ belongs to $A$. This way, $x$ is set to false if and only if the negative literal of $p(x)$ belongs to $A$.
Conversely, any satisfying assignment of $F$ yields a local  adjacency basis for $G$ by first putting
the positive literal of $p(x)$ into the basis if $x$ is set to true by the assignment and by then putting
the negative  literal of $p(x)$ into the basis if $x$ is set to false. Moreover, as the assignment was assumed to be  satisfying, each clause $c$ is satisfied, so that 
at least one of the literal vertices neighboring some outermost vertices of $g(c)$ was put into the basis. Now, two innermost vertices from $g(c)$  could be put into the basis 
(as shown in Fig.~\ref{fig-LocAdjDim clausegadget}) to finally produce a local adjacency basis of size $n+2m$ as required. 

To show NP-hardness for planar instances, we recall Lichtenstein's construction for \textsc{Planar Vertex Cover}; see~\cite{Lic82}: We only have to introduce cycles of length $4m$ instead of paths $P_4$
in our reduction; this enables ``individual'' vertices in the local adjacency generator that correspond to occurrences of literals in the clauses.
Also, the cycle structure of the variables (in Lichtenstein's framework) can be implemented by appropriate interconnections of the variable cycle gadgets.
As $\dimension_{A,l}(C_{4m})=m$, this sketch should suffice to show that \textsc{LocAdjDim} is NP-complete, even when restricted to planar instances.
\end{proof}

\end{document}